\documentclass{amsart}
\usepackage{amssymb}
\usepackage{amsthm}
\usepackage{indentfirst}
\usepackage{amsmath}
\usepackage{amscd}
\usepackage{xypic}

\numberwithin{equation}{section}

\newtheorem{Theorem}{Theorem}[section]

\newtheorem{Proposition}[Theorem]{Proposition}
\newtheorem{Lemma}[Theorem]{Lemma}
\newtheorem{Notation}[Theorem]{Notation}

\newtheorem{Assumption-Notation}[Theorem]{Assumption-Notation}

\newtheorem{Remark}[Theorem]{Remark}
\newtheorem{Corollary}[Theorem]{Corollary}
\newtheorem{Problem}{Problem}

\newtheorem{Claim}[Theorem]{Claim}
\newtheorem{Fact}[Theorem]{Fact}
\newtheorem{Example}[Theorem]{Example}
\newtheorem*{Acknowledgments}{Acknowledgments}

\begin{document}
\title[Surfaces with $p_g = 0, K^2 = 5$]{Surfaces with $p_g = 0$, $K^2 = 5$ and bicanonical maps of degree 4}
\address{Lei Zhang\\School of Mathematics Sciences\\Peking University\\Beijing 100871\\P.R.China}
\email{lzhpkutju@gmail.com}
\author{Lei Zhang}
\maketitle

\textbf{Abstract} Let $S$ be a minimal surface of general type
with $p_g(S) = 0, K_S^2 = 5$ and bicanonical map of degree 4. Denote by $\Sigma$ the bicanonical image. If $\Sigma$ is smooth, then $S$ is a Burniat
surface; and if $\Sigma$ is singular, then we reduced $\Sigma$ to one case and described it, furthermore $S$ has at most one $(-2)$-curve.

\section{Introduction}

Let $S$ be a minimal surface of general type with $p_g(S) = 0$;
denote by $\phi: S \rightarrow \mathbb{P}^{K_S^2} $ its
bicanonical map and by $\Sigma$ the bicanonical image. It is known
that the image $\Sigma$ is a surface for $K_S^2 \geq 2$ (see
\cite{X1}) and $\phi$ is a morphism for $K_S^2 \geq 5$ (see
\cite{M}, \cite{Re}). For $K_S^2 \geq 2$, we denote by $d$ the
degree of $\phi$. Mendes Lopes and Pardini proved: if $K_S^2 = 9$,
then $d=1$; if $K_S^2 = 7,8$, then $d =1,2$; if $K_S^2 = 3,4$ and
$|2K_S|$ is base point free or if $K_S^2 = 5,6$, then $d = 1,2,4$
(see \cite{MP3}, \cite{MP5}). When $d>1$, its image is relatively
simple, so it is possible to describe the surface $S$ precisely.
For more details, we refer the readers to \cite{MP2}, \cite{MP3},
\cite{MP6} and \cite{Par2}. In particular, the surfaces with
$K_S^2 = 6$ and $d = 4$ have been completely characterized: they
are exactly Burniat surfaces (see \cite{MP2}). And in \cite{MP4},
the authors proposed the following problem:
\begin{Problem}\label{Problem} Is it possible to characterize surfaces with $K^2 = 5,p_g =
0$ and $d = 4$? \end{Problem} In this paper, we answer the
problem in the case that the bicanonical image is smooth. Our main
result is:
\begin{Theorem}\label{Theorem A} Let $S$ be a smooth minimal surface of general type
with $p_g(S) = q(S) = 0, K_S^2=5$ and bicanonical map of degree 4.
If $S$ has smooth bicanonical image, then it is a Burniat surface.
\end{Theorem}

\begin{Remark}
For the moduli space of the surfaces described in Theorem \ref{Theorem A}, by Theorem 1.2 in \cite{BC}, it is an irreducible connected component, normal and rational of dimension 3. In fact, Bauer and Catanese studied systematically the moduli spaces of Burniat surfaces with $K^2 = 4,5,6$. For more details, we refer the readers to \cite{BC}.
\end{Remark}

As to the bicanonical image, we proved:
\begin{Theorem}\label{Theorem B}
Let $S$ be as in Theorem \ref{Theorem A}, and denote by $\Sigma$ its bicanonical image. If $\Sigma$ is not smooth, then it is the image of $\psi: \hat{P} \rightarrow \mathbb{P}^5$, where $\hat{P}$ is isomorphic to the blow-up of $\mathbb{P}^2$ at four points $P_1,P_2,P_3,P_4$ such that $P_1, P_2, P_3$ lie on a line, $P_3$ is
infinitely near to $P_2$, $P_1,P_4$ are distinct from $P_2,P_3$, and $\psi$ is defined by the linear system $|-K_{\hat{P}}|$. Moreover, $S$ contains at most one $(-2)$-curve, and the bicanonical map can not lift to a morphism to $\hat{P}$.
\end{Theorem}
\begin{Remark}
By the theorem above, $S$ cannot be constructed by bidouble cover over $\hat{P}$. We failed in constructing a surface as in Theorem \ref{Theorem B}. But if $S$ is such a surface, we get some information about its fibration, and write out a divisor linear equivalent to $K_S$ (see section \ref{2}).
\end{Remark}

The main idea of the proof is from \cite{MP2}. To prove Theorem \ref{Theorem A}, the key step is to show that there exists a
$(-4)$-curve on $S$ which is indicated by the known examples for
this case; to prove Theorem \ref{Theorem B}, we develop a method to make use of the ramification divisor (Lemma \ref{ramification}) and refer to computer sometimes as the calculation is very complicate. 

The plan of the paper is: in Section \ref{2}, we collect some basic tools; in Section \ref{construction}, \ref{divisor}, \ref{main} and \ref{prop}, we consider the case when the bicanonical image is smooth and prove Theorem \ref{Theorem A}; in Section \ref{bicimage}, we study the the case when the bicanonical image is singular and prove Theorem \ref{Theorem B}.

\textbf{Notations and conventions:} We work over complex numbers;
all varieties are assumed to be compact and algebraic. We don't
distinguish between the line bundles and the divisors on a smooth
variety, and we use both the additive and multiplicative notation.
We say a line bundle is effective if it has a global non-zero
section. Let $V$ be a prime Weil divisor and let $D$ be another
divisor on a normal $\mathbb{Q}$-factorial projective variety.
Then $ord_V(D)$ denotes the vanishing order of an equation for $D$
in the local ring along the subvariety $V$. Let $D_1$ and $D_2$ be
two $\mathbb{Q}$-Cartier~
divisor, $D_1 \leq
D_2$ means $D_2 - D_1$ is an effective $\mathbb{Q}$-Cartier~
divisor; moreover if $D_1, D_2$ are Cartier, $D_1 \equiv D_2$
means they are linearly equivalent. Let $f: X \rightarrow Y$ be a
morphism between two normal projective varieties. Let $D$ be a
$\mathbb{Q}$-Cartier divisor on $Y$ and assume $mD$ is Cartier
for a positive integer $m$. Then $f^*D := \frac{1}{m}f^*mD$.
Write $D = \sum_ia_iD_i$ where $a_i \in \mathbb{Q}$ and $D_i$
is a reduced Weil divisor. Then $[D] := \sum_i[a_i]D_i$. Denote by $S$
the singular locus of $Y$ and by $i: Y-S \hookrightarrow Y$
the natural inclusion map. If moreover $D$ is a Weil divisor, then $\mathcal{O}_Y(D):=
i_*\mathcal{O}_{Y-S}(D)$, i.e, $\mathcal{O}_Y(D)_y = \{s \in K(Y):
(s^m) + mD \geq 0~in ~a ~neighbourhood ~of ~y\}$.
\begin{Acknowledgments} I am grateful to Prof. Jinxing Cai, Wenfei Liu and Yifan
Chen for their valuable suggestions and many useful discussions
during the preparation of this paper. Especially Wenfei Liu spent
much time improving my English. I also thank Prof.
Rita Pardini for her useful communications. And I would like to
express my appreciation to the anonymous referee for his/her
suggestions making the proof of some lemmas briefer and more
complete.
\end{Acknowledgments}

\section{Preliminaries}\label{2}

The main tools we use in this paper are from \cite{MP2}. And we
list them in the following. Let $S$ be a smooth surface, let $D
\subset S$ be a curve having at most negligible singularities (possibly empty) and let $M$ be a
line bundle on $S$ such that $2M \equiv D$. Then there exists a
double cover $\pi: Y \rightarrow S$ branched over $D$ and such
that $\pi_*\mathcal{O}_Y = \mathcal{O}_S \bigoplus M^{-1}$. Note
that $Y$ is smooth if $D$ is smooth, $Y$ has at most canonical singularities if $D$ has at most negligible singularities, and $Y$ is connected if and
only if $M$ is non-trivial. The invariants of Y can be calculated
as follows:
\begin{equation}\label{tool}
\begin{split}
K_Y^2 &= 2(K_S + M)^2 \\
\chi(\mathcal{O}_Y) &= 2 \chi(\mathcal{O}_S)+ \frac{1}{2}M(K_S +
M) = 2 \chi(\mathcal{O}_S)+ \frac{1}{2}\frac{D}{2}(K_S +
\frac{D}{2})\\
p_g(Y) &= p_g(S) + h^0(S, K_S + M)
\end{split}
\end{equation}
\begin{Proposition}[Proposition 2.1 in \cite{MP2}]\label{A.i} Let S be a smooth surface
with $p_g(S) = q (S) =0$, and let $\pi: Y \rightarrow S$ be a
smooth double cover. Suppose that $q(Y) > 0$. Denote by $\alpha: Y
\rightarrow A$ be the Albanese map.  Then
\begin{enumerate}
\item[i)]{the Albanese image of $Y$ is a curve $B$;}
\item[ii)]{there exists a fibration $g: S \rightarrow \mathbb{P}^1$ and a
degree $2$ map $p: B \rightarrow \mathbb{P}^1$ such that
$p\circ\alpha = g\circ\pi$.}
\end{enumerate}
\end{Proposition}
\begin{Proposition}[Corollary 2.2 in \cite{MP2}]\label{B.i} Let S be a smooth surface
of general type with $p_g(S) = q(S) =0, K_S^2 \geq 3$, and let
$\pi: Y \rightarrow S$ be a smooth double cover. Then $K_Y^2 \geq
16(q(Y) - 1)$.
\end{Proposition}
\begin{Proposition}[Remark 2.3 in \cite{MP2}]\label{C.i} Let S be a smooth surface
with $p_g(S) = q (S) =0$, and let $\pi: Y \rightarrow S$ be a
smooth double cover. Let $g: S \rightarrow \mathbb{P}^1$ be a
fibration such that the general fiber of $g\circ\pi$ is not
connected. Let $\pi'\circ g': Y \rightarrow B \rightarrow
\mathbb{P}^1$ be the Stein factorization. Then one has the
following commutative diagram:
\[\begin{CD}
Y       @>\pi>>       S \\
@Vg'VV               @VgVV \\
B       @>\pi'>>    \mathbb{P}^1
\end{CD} \]
where $B$ is a smooth curve of genus $b$ and $\pi'$ is a double
cover. Furthermore if $k$ is the cardinality of the image in
$\mathbb{P}^1$ of the branch locus of $\pi$, then $g$ has at least
$2b + 2 -k$ fibers that are divisible by $2$.
\end{Proposition}

Here we make a useful remark since we often have to deal with singular double cover.
\begin{Remark}
If the branch locus has negligible singularities which causes that $Y$ has canonical singularities, then the propositions above still hold true.
\end{Remark}

The following lemma plays an important role in the proof of Theorem \ref{Theorem B}, it can be proved by using the idea of the proof Lemma 5.7 in \cite{MP2}, and we omit the details.
\begin{Lemma}\label{pnf}Let $\pi: X \rightarrow Y$ be a finite morphism
between two normal $\mathbb{Q}$-factorial surfaces. Assume that
$X$ and $Y$ have equal Picard numbers. Let $u: X \rightarrow C$ of
$X$ be a fibration. Then there exists a finite morphism $\pi': C
\rightarrow B$ and a fibration $v: Y \rightarrow B$ satisfying the
following commutative diagram
\[\begin{CD}
X       @>\pi>>       Y \\
@VuVV               @VvVV \\
C       @>\pi'>>     B
\end{CD} \]
\end{Lemma}

\begin{Lemma}\label{pullback} Let
$f: X \rightarrow Y$ be a generically finite surjective morphism
between two normal projective varieties. Let $D \subset Y$ and $G \subset X$ be two $\mathbb{Q}$-Cartier Weil divisors
such that $G \geq [f^*D]$. Then
$h^0(\mathcal{O}_X(G)) \geq h^0(\mathcal{O}_Y(D))$.
\end{Lemma}
\begin{proof}
Note that $H^0(\mathcal{O}_Y(D)) \cong \{s \in K(Y)| (s) + D
\geq 0\}$. It suffices to show that $(f^*s) + [f^*D] \geq 0$ for
every $s \in K(Y)$ such that $(s) + D \geq 0$. Assume $mD$ is Cartier where $m$ is a positive integer.
By definition we
have $(s^m) + mD \geq 0$, thus $(f^*s^m) + f^*mD = m((f^*s) +
f^*D) \geq 0$. Since $(f^*s)$ is Cartier, $(f^*s) + [f^*D] \geq
0$, so the lemma follows.
\end{proof}
Lemma \ref{pullback} is very useful, we can give a low bound of $h^0(\mathcal{O}_X(G))$ by calculating $h^0(\mathcal{O}_Y(D))$. The divisor that we have to deal with is often not Cartier.

The following lemma is well known to experts, but we give detailed proof
here for lack of references.

\begin{Lemma}\label{cvroffbr}
Let $f:S \rightarrow \Delta$ be a fibration over the unit disc $\Delta$ such that $S_0 = 2M$ be the only singular fiber. Let $\pi: X \rightarrow S$ be an etale double cover given by the relation $2L \equiv \mathcal{O}_S$. Let $\Delta' \rightarrow \Delta$ be a double cover given by $t \rightarrow s^2$ branched along the point $0$. If $X$ coincides with the normalization of the fiber product $\Delta' \times_{\Delta}S$, then $L \equiv M$.
\end{Lemma}
\begin{proof}
 Denote by $\pi': X' \rightarrow S$ the double cover given by the relation $2M \equiv \mathcal{O}_S$. Considering the Stein factorizations of the composition map $f\circ \pi'$ and $f\circ \pi$, we find that both $X$ and $X'$ are isomorphic to the normalization of the product $S\times _{\Delta}\Delta'$, hence the double cover $\pi$ coincides with $\pi'$, and thus $\pi_*\mathcal{O}_X \cong \pi_*\mathcal{O}_{X'}$. So the lemma follows from the fact that $\pi_*\mathcal{O}_X \cong \mathcal{O}_S \oplus \mathcal{O}_S(L^{-1})$ and $\pi_*\mathcal{O}_X \cong \mathcal{O}_S \oplus \mathcal{O}_S(M^{-1})$.
\end{proof}

\begin{Lemma}\label{ramification}
Let $h: X \rightarrow T$ be a generically finite morphism between
two normal surfaces. Let $e \subset T$ be a reduced and irreducible $\mathbb{Q}$-Cartier Weil
divisor such that $e^2 < 0$. Denote by $R$ be the ramification divisor,
and let $R'$ be an effective divisor such that $R' \leq R$. Then we have
$R'(h^*e)
> (h^*e)^2$.
\end{Lemma}
\begin{proof}
Let $g \circ \eta: X \rightarrow Y \rightarrow T$ be the
factorization of $h$. Write $g^*e = \sum_ia_iE_i$. Note that
$(g_*E_i)e < 0$ since $e^2 <0$. Then we have
\begin{equation}
\begin{split}
R'(h^*e)
&= (\eta_* R')( g^*e )\geq (\sum_i(a_i - 1)E_i)(g^*e)\\
&= (g^*e - \sum_iE_i)g^*e \\
&= (g^*e)^2 - \sum_i(g_*E_i)e >(g^*e)^2 = (h^*e)^2
\end{split}
\end{equation}
\end{proof}

\section{Examples}\label{construction}

In this section we study some examples and make a useful
observation. Let $\rho: \Sigma\rightarrow \mathbb{P}^2$ be the
blow-up of $\mathbb{P}^2$ at $4$ points $P_1,P_2,P_3,P_4$ in
general position. We denote by $l$ the pull-back of a line on
$\mathbb{P}^2$, by $e_i$ the exceptional curve corresponding to
$P_i, i=1,2,3,4$, by $e_i'$ the strict transform of the line
joining $P_j$ and $P_k$ where $\{i,j,k\}=\{1,2,3\}$, by $g_i$ the
strict transform of the line joining $P_4$ and $P_i, i=1,2,3$, and
by $m_i$ the strict transform of a general line through $P_i,
i=1,2,3,4$. Then the Picard group of $\Sigma$ is a free Abelian
group generated by the classes of $l, e_1, e_2, e_3, e_4$. The
anticanonical class $ -K_\Sigma \equiv 3l - e_1 - e_2 - e_3 - e_4$
is very ample. And the linear system $|-K_\Sigma|$ embeds $\Sigma$
as a smooth surface of degree $5$ in $\mathbb{P}^5$, the so-called
del Pezzo surface of degree $5$.

As to the examples for the surfaces of general type with $p_g=0,
K^2 = 5$ and bicanonical map of degree 4, there are two known
examples: example 6 in \cite{Cat}, namely, Burniat surface (see
\cite{Pet}, \cite{Bu}) and example 7 in \cite{Cat}. We describe
the two examples by giving a collection of bidouble cover data
over $\Sigma$. Let $\Gamma = \mathbb{Z}_2 \bigoplus \mathbb{Z}_2$,
denote by $\gamma_1,\gamma_2,\gamma_3$ the nonzero elements of
$\Gamma$ and by $\chi_i\in\Gamma^*$ the nontrivial character
orthogonal to $\gamma_i$; by \cite{Par1}, to define a smooth
$\Gamma$-cover $\phi : S \rightarrow \Sigma$ one assigns the
following bidouble cover data:

$i)$ smooth divisors $D_i, i=1,2,3$ such that $D = D_1 + D_2 +
D_3$ is a normal crossing divisor,

$ii)$ line bundles $L_i, i=1,2,3$ satisfying $2L_i \equiv D_k +
D_j, \{i,j,k\}=\{1,2,3\}.$\\
Note that once one assigns the data $i)$, the data $ii)$ are
determined since the Picard group of $\Sigma$ is free.
\begin{Example}[Example 6 in \cite{Cat}, i.e., Burniat sufaces]\label{expl1} We construct a surface $S$ by
giving the following bidouble cover data:
\begin{equation*}
\begin{split}
&D_1  = e_3 + e_3' + g_1 + m_1 \equiv 3l - 3e_1 - e_2 + e_3 - e_4\\
&D_2  = e_1 + e_1' + g_2 + m_2 \equiv 3l + e_1  - 3e_2 - e_3 - e_4 \\
&D_3  = e_2 + e_2' + g_3 + m_3 \equiv 3l - e_1 + e_2 - 3e_3 - e_4
\end{split}
\end{equation*}
We denote by $\pi: S \rightarrow \Sigma $ the $4$-to-$1$ covering
map and by $\mathcal{S}$ the family of surfaces constructed in the
example.
\end{Example}
\begin{Example}[Example 7 in \cite{Cat}]\label{expl2} We construct a surface $S'$ by giving the following
bidouble cover data:
\begin{equation*}
\begin{split}
&D_1 = m_1 + e_2 + e_3 + e_4 \equiv l - e_1 + e_2 + e_3 + e_4\\
&D_2 = e_1 + e_1' + g_2 + m_2 \equiv 3l + e_1  - 3e_2 - e_3 - e_4 \\
&D_3 = e_2' + g_1 + g_3 + B \equiv 5l - 3e_1 - e_2 - 3e_3 - 3e_4
\end{split}
\end{equation*}
where $B$ is the strict transform of a conic passing through the
points $P_1, P_2, P_3, P_4$. We denote by $\pi: S' \rightarrow
\Sigma$ the $4$-to-$1$ covering map and by $\mathcal{S}'$ the
family of surfaces constructed in the example.
\end{Example}
\begin{Remark} For a surface $S$ in the two families constructed above,
by $\pi_*\mathcal{O}_S \cong \mathcal{O}_{\Sigma}
\oplus_iL_i^{-1}$, we calculate that $p_g(S) = q(S) = 0$. Since $D \equiv -3K_\Sigma$, by the formula $2K_S \equiv \pi^*(2K_\Sigma + D)$ (cf. \cite{Par1}), we get $2K_S \equiv \pi^*(-K_{\Sigma})$, thus $K_S^2 = 5$.
Checking that $h^0(\Sigma, -K_{\Sigma}\otimes L_i^{-1}) = 0$ for
$i=1,2,3$, we have $H^0(S, 2K_S) \cong H^0(\Sigma,
-K_{\Sigma})\oplus_i H^0(\Sigma, -K_{\Sigma}\otimes L_i^{-1})
\cong H^0(\Sigma, -K_{\Sigma})$. This implies that the bicanonical
map of $S$ coincides with the covering map $\pi$.\end{Remark}

Although the two collections of bidouble cover data are of
different forms, we have following observation:
\begin{Claim}\label{identifying} The two families $\mathcal{S}$ and $\mathcal{S}'$
are the same.\end{Claim} To see this, first note that $\Sigma$ has
two types automorphisms.

Type 1: An automorphism of this type arises
from a linear automorphism of $\mathbb{P}^2$. Let $s \in S_4$ be a permutation
of $\{1,2,3,4\}$. Then there exists a unique linear automorphism
$\bar{\eta_s}$ of $\mathbb{P}^2$ such that $\bar{\eta_s}(P_i) =
P_{s(i)}$ and a unique automorphism $\eta_s$ of $\Sigma$ fitting
in the following commutative diagram:
\[\begin{CD}
\Sigma               @>\eta_s>>       \Sigma  \\
@V\rho VV                             @V\rho VV \\
\mathbb{P}^2       @>\bar{\eta_s}>>         \mathbb{P}^2
\end{CD} \]
It follows from the definition that $\eta_s^*e_{s_i} =
e_i~for~i=1,2,3,4$ and $\eta_s^*l \equiv l.$

Type 2: An automorphism of type 2 arises from a birational map
from $\mathbb{P}^2$ to itself, the so-called Cremona
transformation (cf. \cite{Har} P.30). Here we give an example.
Denote by $(x_0, x_1, x_2)$ the homogenous coordinate of
$\mathbb{P}^2$. Up to a linear transformation, we can assume the
coordinates of $P_i, i=1,2,3,4$ are as follows: $P_1(1,0,0), P_2
(0,1,0),P_3(0,0,1), P_4(1,1,1)$. Note that for $i=1,2,3$, $e_i'$
is the strict transform of the line defined by $x_{i-1} = 0$. A
Cremona transformation $\varphi: \mathbb{P}^2 \rightarrow
\mathbb{P}^2$ is given by $\varphi: (x_0,x_1,x_2)\mapsto
(x_1x_2,x_0x_2,x_0x_1)$, so it is well defined on $\mathbb{P}^2
\setminus \{P_1,P_2,P_3\}$, and it fixes $P_4$. By Example 4.2.3 chap.V in
\cite{Har}, to extend $\varphi$ to a morphism, it suffices to
blow up $\mathbb{P}^2$ at the points $P_1,P_2,P_3$. So $\varphi$
extends to a morphism $\rho': \Sigma \rightarrow \mathbb{P}^2$
which blows down the curves $e_i',i=1,2,3$ and $e_4$ to the points
$P_i, i=1,2,3$ and $P_4$ respectively. Note that both $\rho$ and $\rho'$ are the blowing-up of $\mathbb{P}^2$ at the points $P_i, i =1,2,3,4$. Thus the map $\rho'$ lifts
to a morphism $\tau:\Sigma \rightarrow \Sigma$ fitting into the
following diagram:

\centerline{\xymatrix{
&\Sigma   \ar[r]^\tau \ar[d]_\rho \ar[dr]^{\rho'}       &\Sigma \ar[d]^\rho\\
&\mathbb{P}^2   \ar@{.>}[r] ^\varphi     &\mathbb{P}^2
}}
We can see that $\tau$ is an automorphism of $\Sigma$. And it follows from the definition that $\tau^*e_i' = e_i, \tau^*e_i =
e_i',i=1,2,3$ and $\tau^*e_4 = e_4$, hence $\tau^*l \equiv
\tau^*(e_1' + e_2 + e_3) \equiv 2l - e_1 - e_2 - e_3$. For a
permutation $s$ of $\{1,2,3,4\}$, let $\eta_s$ be defined as in
Type 1. Then the automorphism $\eta_s^{-1} \circ \tau \circ \eta_s$ also arises
from a Cremona transformation, and it fixes the curve $\eta_s^{-1}(e_4)$.

\begin{proof}[\bf Proof of Claim \ref{identifying}]
Let $S'\in \mathcal{S}'$ be a surface given by the data in Example
\ref{expl2}. Consider an automorphism $\tau$ of type $2$
introduced above such that
$$\tau^*e_i' = e_i, \tau^*e_i = e_i'~for~i=1,2,3, ~\tau^*e_4 = e_4,
\tau^*l \equiv 2l - e_1 - e_2 - e_3.$$ It follows that
$$\tau^*g_i
\equiv l - e_i - e_4, ~\tau^*m_i \equiv l -e_i
~for~i=1,2,3,~~\tau^*B \equiv l - e_4.$$ Let $n_i = \tau^*m_i$ for
$ i = 1,2,3$ and $n_4 = \tau^*B$. Then $n_i$ is the strict
transform of a line through $P_i, i = 1,2,3,4$. We obtain
\begin{equation}
\begin{split}
&\tau^*D_1 = n_1 + e_2' + e_3' + e_4 \equiv 3l - 3e_1 - e_2 - e_3 + e_4\\
&\tau^*D_2 = e_1 + e_1' + g_2 + n_2  \equiv 3l + e_1  - 3e_2 - e_3 - e_4 \\
&\tau^*D_3 = e_2 + g_1 + g_3  + n_4 \equiv 3l - e_1 + e_2 - e_3 -
3e_4.
\end{split}
\end{equation}

Let $s = (34)$, and let $\eta_s$ be an automorphism of type 1
defined as above. Then we have
$$\eta_s^*e_1 = e_1, \eta_s^*e_2 = e_2, \eta_s^*e_3 = e_4, \eta_s^*e_4 = e_3, \eta_s^*l \equiv l.$$
It follows that
\begin{equation}
\begin{split}
&\eta_s^*\tau^*D_1 = k_1 + g_1 + e_3' + e_3 \equiv 3l - 3e_1 - e_2 + e_3 - e_4\\
&\eta_s^*\tau^*D_2 = e_1 + g_2 + e_1' + k_2  \equiv 3l + e_1  - 3e_2 - e_3 - e_4 \\
&\eta_s^*\tau^*D_3 = e_2 + e_2' + g_3  + k_3 \equiv 3l - e_1 + e_2
- 3e_3 - e_4
\end{split}
\end{equation}
where $k_1 = \eta_s^*n_1 \equiv l - e_1, ~k_2 = \eta_s^*n_2 \equiv
l - e_2,~k_3 = \eta_s^*n_4 \equiv l - e_3$, and for $i = 1,2,3$,
$k_i$ is the strict transform of a line through $P_i$.

The data pulled back via $\tau \circ \eta_s$ is exactly one
collection of bidouble cover data as in Example \ref{expl1}.
Therefore the claim follows.
\end{proof}
There are exactly 10 $(-1)$-curves on $\Sigma$ in all. For
every $(-1)$-curve, there are 3 $(-1)$-curves intersecting
it and 6 $(-1)$-curves not intersecting it. An automorphism of $\Sigma$ sends a $(-1)$-curve to a $(-1)$ -curve, so it induces an action on the set of $(-1)$-curves on $\Sigma$. Here we list some facts about the existence of the automorphism with a certain action on the set of $(-1)$-curves.

\begin{Fact}\label{symofcurves} Let all the notations be as above. Then we have
\begin{enumerate}
\item[i)]{For two $(-1)$-curves $c_1,c_2$ on $\Sigma$, there exists an
automorphism $\sigma$ of $\Sigma$ such that $\sigma(c_1) = c_2$;}
\item[ii)]{For three $(-1)$-curves $c_1,c_2, c$ on $\Sigma$ such that $c_1c =c_2c
=0$, there exists an automorphism $\sigma$ of $\Sigma$ such that
$\sigma(c) = c$ and $\sigma(c_1) = c_2$;}
\item[iii)]{For four $(-1)$-curves $c_1,c_2,c_3,c_4$ on $\Sigma$ such that $c_1c_2
= c_3c_4 = 0$, there exists an automorphism $\sigma$ of $\Sigma$ such that
$\sigma(c_1) = c_3$ and $\sigma(c_2) = c_4$.}
\end{enumerate}
\end{Fact}
\begin{proof} In the following, when we say a curve is exceptional, we mean it is $\rho$-exceptional.

$i)$ If necessary, acting on $c_1$ by a suitable automorphism of type 2, we can assume $c_1$ is an exceptional curve, say
$c_1 = e_1$. If $c_2$ is also exceptional, then there exists an
automorphism of type 1 sending $c_1$ to $c_2$. Now we assume $c_2$ is
not exceptional. If $c_2$ does not intersect $c_1$, say $c_2 =
e_1'$, then $\tau(e_1) = e_1' = c_2$; if $c_2$ intersects $c_1$,
say $c_2 = e_2'$, then $\eta_{(12)}\circ \tau(e_1) = \eta_{(12)}(e_1') = e_2' = c_2$.

$ii)$ Thanks to $i)$, we can assume $c = e_4$. So $c_1, c_2$ belong to the set $\{e_i,e_i'\}_{i = 1,2,3}$. Then arguing as in $i)$, we prove
$ii)$.

$iii)$ By $i)$, we can find
an automorphism $\sigma_1$ such that $\sigma_1(c_1) = c_3$, then
we have $\sigma_1(c_2)\cap c_3 = \sigma_1(c_2 \cap c_1) = \phi$. And since $c_4 \cap c_3
=\phi$, applying $ii)$ gives an automorphism fixing $c_3$ and sending $\sigma_1(c_2)$ to $c_4$, then $iii)$ follows.
\end{proof}
\begin{Remark}
Notice that in Example~\ref{expl1}, $e_4$ is the only $(-1)$-curve
that is not contained in the branch divisor, and the pull-back $\pi^*e_4$
is a reduced $(-4)$-curve while the pull-backs of the other
$(-1)$-curves are non-reduced. This observation suggests that we
should prove that $S$ has a $(-4)$-curve.
\end{Remark}

\section{Divisors, pencils and torsion of $S$}\label{divisor}
This section is the preparation for the proof of the main theorem.
\begin{Notation}\label{not1} Let $S$ be a surface with $p_g(S)=q(S)=0,
K_S^2=5$. Let $\phi : S \rightarrow \mathbb{P}^5$ be the
bicanonical map and $\Sigma$ be the bicanonical image. We assume
the degree of $\phi$ is 4 and $\Sigma$ is a smooth. By \cite{Na},
$\Sigma$ is a del Pezzo surface of degree $5$ in $\mathbb{P}^5$.
Let $\rho: \Sigma\rightarrow \mathbb{P}^2$ be the blow-up of
$\mathbb{P}^2$ at $4$ points $P_1,P_2,P_3,P_4$ in general
position, and let $l,e_i,e_j',g_j,i=1,2,3,4,j=1,2,3$ be as at
the beginning of Section \ref{construction}.
Set $f_i \equiv l - e_i$ and $F_i \equiv \phi^*(l -
e_i)$ for $i = 1,2,3,4$.
\end{Notation}

\begin{Proposition} Let the notations be as in \ref{not1}. For $i = 1,2,3,4$, if $f_i \in |f_i|$ is general, then $\phi^*f_i$ is connected, hence $|F_i|$ induces a genus 3 fibration $u_i: S\rightarrow \mathbb{P}^1$.
\end{Proposition}
\begin{proof}
By Bertini theorem, for a general element $f_i \in |f_i|$, $F_i =
\phi^*f_i$ is smooth. Since $K_SF_i = 4$, it suffices to prove
$\phi^*f_i$ is connected. Otherwise, we will get a genus 2
fibration of $S$. However, by [\cite{X1} P.37], $S$ has no genus 2
fibration. Then we get a contradiction.
\end{proof}
\begin{Proposition}\label{infofpb}
Let the notations be as in \ref{not1}. Then the bicanonical map $\phi$ is finite, and for $i =1,2,3,4$, the pull-back of an irreducible curve in $|f_i|$ is also irreducible (possibly
non-reduced).
\end{Proposition}
\begin{proof}
By $\chi(S) = 1$ and $K_S^2 =5$, Noether's formula gives $e(S) = 7$. Then by $p_g(S) = q(S) = 0$, we get  $h^2(S) = h^2(\Sigma) = 5$. So $\phi^*: H^2(\Sigma, \mathbb{Q}) \rightarrow H^2(S, \mathbb{Q})$ is an isomorphism preserving the intersection form up to multiplication by 4. Therefore
the bicanonical map $\phi$ is finite. For an irreducible curve $f_1 \in |f_1|$, if $\phi^*f_1$ is reducible,
then it contains an irreducible component $C$ with $C^2 < 0$. Put $D = C- \frac{C(\phi^*e_1)}{4}\phi^*f_1$. Then $D^2 = C^2 <0$, and $D(\phi^*e_1) = 0$. And for $i=2,3,4$, $(C- \frac{C(\phi^*e_1)}{4}\phi^*f_1)\phi^*e_i = 0$ since $e_i$ is contained in one fiber of the pencil $|l-e_1|$. Then we can see that the intersection matrix of
$\phi^*l, C- \frac{C(\phi^*e_1)}{4}\phi^*f_1, \phi^*e_i,i=1,2,3,4$ has rank 6. However, this contradicts $h^2(S) = 5$, thus $\phi^*f_1$ is irreducible. The proof for the other cases is similar.
\end{proof}

\begin{Lemma}\label{B.p} Let $\phi : S
\rightarrow T$ be a finite morphism between two smooth surfaces.
Let $h$ be a divisor on $T$ such that $|\phi^*h| = \phi^*|h|$.
Let $M$ be a divisor on $S$ such that the linear system $|M|$ has
no fixed part. Suppose that $\phi^*h - M$ is effective.
Then there exists a divisor $m \subset T$ such that $|M|=
\phi^*|m|$. Furthermore the line bundle $h - m$ is effective.
\end{Lemma}
\begin{proof}
Choose an effective divisor $N$ on $S$ such that $N\equiv \phi^*h - M$. Let $M_1
\in |M|$ be a general element. Then there exists an element $h_1 \in |h|$ such that
$M_1 + N = \phi^*h_1$ since $|\phi^*h| = \phi^*|h|$. We can assume $\phi(M_1)$ and $\phi(N)$ have no common components. Write $h_1 = m_1 + n$ where $m_1$ and $n$ are supported on $\phi(M_1)$ and $\phi(N)$ respectively.
Since $\phi$ is a finite morphism, $\phi^*m_1$ and $\phi^*n$ have no common components,
so $\phi^*m_1 = M_1$ and $\phi^*n = N$. Since $M_1$ is general, we conclude that $|M| = \phi^*|m_1|$.
Obviously $h - m_1$ is effective, and we are done.
\end{proof}
\begin{Lemma}\label{subdivisor} There does not exist a divisor $d$ on
$\Sigma$ such that $h^0(\Sigma, d) > 1$ and that the line bundle
$-K_{\Sigma} - 2d$ is effective.
\end{Lemma}
\begin{proof}
To the contrary, suppose that there exists such a divisor $d$. Assume
$d \equiv al - b_1e_1 - b_2e_2 - b_3e_3 - b_4e_4$. The condition
that $-K_{\Sigma} - 2d$ is effective implies that $a \leq 1$. And
$h^0(\Sigma, d)
> 1$ implies that $a \geq 1$. So $a =1$, and at most one of $b_1, b_2,b_3,b_4$ is positive.
But then the line bundle $-K_{\Sigma} - 2d \equiv l - (1-b_1)e_1 -
(1-b_2)e_2 - (1-b_3)e_3 - (1-b_4)e_4$ can not be effective since
there is no line on $\mathbb{P}^2$ passing through 3 points in general
position. Thus the lemma is true.
\end{proof}
\begin{Lemma}\label{A.prop} Let $S$ and $\Sigma$ be as in \ref{not1}. Let $D \subset S$ be a divisor. If there
exists a divisor $d$ on $\Sigma$ such that
\begin{enumerate}
\item[i)]{$\phi^*d \equiv 2D$;}
\item[ii)]{the line bundle $-K_{\Sigma} - d$ is effective,}
\end{enumerate}
then $h^0(S, D) \leq 1$.
\end{Lemma}
\begin{proof}
To the contrary, suppose that $h^0(S, D) > 1$. We may write $|D| =
|M| + F$ where $|M|$ is the moving part and $F$ is the fixed part.
Since $|2K_S| = |\phi^*(-K_{\Sigma})| = \phi^*|-K_{\Sigma}|$ and
$\phi^*(-K_{\Sigma}) - M > \phi^*(-K_{\Sigma} - d)$ is effective, applying Lemma \ref{B.p} gives a divisor $m$ on $\Sigma$ such that $\phi^*|m| = |M|$. Choose an element $M_1 \in |M|$ and an effective divisor $N$ on $S$ such that $2M_1 + 2F + N \equiv \phi^*(-K_{\Sigma})$. Note that we can find $h \in|-K_\Sigma|$ and $m_1 \in |m|$ such that $2M_1 + 2F + N = \phi^*h$ and $2M_1 = \phi^*(2m_1)$. So we conclude that $h - 2m_1$ is effective, i.e.,
the line bundle $-K_{\Sigma} - 2m$ is effective. However, this contradicts
Lemma~\ref{subdivisor}.
\end{proof}

Now we analyze the pull-backs of the $(-1)$-curves on $\Sigma$
which give much information about the surface $S$. We begin with
one lemma from \cite{MP2}.
\begin{Lemma}[Lemma 5.1 in \cite{MP2}]\label{A.d} Let $\phi : S
\rightarrow \Sigma$ be as in \ref{not1}, and let $C\subset\Sigma$
be a $(-1)-curve$. Then we have either
\begin{enumerate}
\item[i)] {$\phi^*C$ is a reduced smooth rational $(-4)$-curve;
or} \item[ii)] {$\phi^*C = 2E$ where $E$ is an irreducible curve
with $E^2 = -1, K_SE = 1$.}
\end{enumerate}
\end{Lemma}
\begin{Lemma}\label{B.d} Let $S$ be as in \ref{not1}.
Then there are at most two disjoint $(-4)$-curves on $S$.
\end{Lemma}
\begin{proof}
Let $r$ be the cardinality of a set of smooth disjoint rational
curves with self-intersection $-4$. Then by \cite{Miy}, one has
$r\cdot\frac{25}{12} \leq c_2(S) - \frac{1}{3}K_S^2 =
\frac{16}{3}$, therefore $r \leq 2$.
\end{proof}
\begin{Proposition}\label{mainprop} Let $S$ and $\Sigma$ be as in
\ref{not1}. Then there exists at least one $(-1)$-curve on
$\Sigma$ such that its pull-back is a $(-4)$-curve.
\end{Proposition}
As the proof of the proposition is long, we postpone the proof to
the last section.
\begin{Proposition}\label{mainprop2} Let the notations be as in \ref{not1}, then there
do not exist two $(-1)$-curves $C_1, C_2 \subset \Sigma$
satisfying
\begin{enumerate}
\item[i)] {$C_1 \cap C_2 = \emptyset$;}
\item[ii)] {both $\phi^*C_1$ and $\phi^*C_2$ are $(-4)$-curves.}
\end{enumerate}
\end{Proposition}
\begin{proof}
Suppose on the contrary that the proposition is not true. Thanks to
Fact \ref{symofcurves} $iii)$, we
may assume $C_1 = e_4$ and $C_2= e_2$. Then both $E_2 = \phi^*e_2$ and $E_4
= \phi^*e_4$ are reduced rational $(-4)$-curves. By Lemma
\ref{B.d}, $\phi^*e_2'$ is non-reduced. And by Lemma \ref{A.d} we
may write $\phi^*e_2'= 2E_2'$. From the formula
$$2K_S \equiv \phi^*(3l - e_1 - e_2 - e_3 - e_4) \equiv
\phi^*(2g_2 + e_2' + e_2 + e_4) \equiv 2\phi^*(g_2) + 2E_2' + E_2
+ E_4,$$ we obtain $2(K_S - \phi^*(g_2) - E_2') \equiv E_2 + E_4$.
Therefore, we get a double cover $\pi: Y\rightarrow S$ branched
over $E_2$ and $E_4$. Applying formula \ref{tool}, we have
$$\chi(\mathcal{O}_Y) = 2 + \frac{(K_S - \phi^*(g_2) - E_2')\cdot (2K_S -
\phi^*(g_2) - E_2')}{2} = 2,$$
$$K_Y^2 = 2(2K_S -
\phi^*(g_2) - E_2')^2 = 14 ,$$
$$p_g(Y) = h^0(S, 2K_S - \phi^*(g_2) - E_2') = h^0(S, \phi^*(g_2)
+ E_2' + E_2 + E_4)$$
$$= h^0(S, \phi^*(g_2 + e_2 + e_4) + E_2') = h^0(S, \phi^*(l) +
E_2') \geq 3.$$ Hence $q(Y) \geq 2$ and $K_Y^2 < 16(q(Y) - 1)$.
This contradicts Proposition \ref{B.i}.
\end{proof}
Since there do not exist three $(-1)$-curves on $\Sigma$
intersecting each other, by the propositions above, we get the
following corollary.
\begin{Corollary} There are at most two $(-1)$-curves
on $\Sigma$ whose pull-back is a $(-4)$-curve, and if there are
two such curves, then they intersect.
\end{Corollary}
\begin{Assumption-Notation}\label{not2} By Fact \ref{symofcurves} $i)$ and the corollary above,
up to an automorphism of $\Sigma$, we may assume $\phi^*e_4 = E_4$
where $E_4$ is a $(-4)$-curve, $\phi^*e_i = 2E_i, \phi^*e_i' =
2E_i',i = 1,2,3$ and $\phi^*g_j = 2G_j,j=1,2$. With these
assumptions, we know that $2(E_j + E_k')$ and $2(E_j' + E_k)$ are
two double fibers of $u_i: S \rightarrow \mathbb{P}^1$ where
$\{i,j,k\} = \{1,2,3\}$. Set $\eta_i \equiv (E_j + E_k') -(E_j' +
E_k)$ where $\{i,j,k\} = \{1,2,3\}$, and set $\eta \equiv K_S - \sum_{j=1}^{j=3} (E_j + E_j')$. Then $2\eta \equiv -E_4$; and \cite{BPV} Lemma 8.3 chap.III gives that $\eta_i \neq 0$ for $i =1,2,3$, hence it is torsion of
order $2$.
\end{Assumption-Notation}
The following proposition corresponds to Proposition 5.9 in
\cite{MP2}. For the readers' convenience, we will give an explicit
proof.
\begin{Proposition}\label{A.p} Let notations and assumptions be as in \ref{not1} and \ref{not2},
and let $F_i \in |F_i|,i=1,2,3$ be a general curve. Then
$F_j\mid_{F_i} \equiv K_{F_i}~if~i\neq j.$
\end{Proposition}
\begin{proof}
We show that $F_2\mid_{F_1} \equiv K_{F_1}$. Note that $2K_S
\equiv F_1 + 2(2E_1 + E_3' + E_2') - E_4$, then we have
$$2(K_S - (2E_1 + E_3' + E_2') + E_4) \equiv F_1 + E_4.$$ Given
this relation, we get a double cover $\pi: Y\rightarrow S$
branched over $F_1$ and $E_4$. By formula \ref{tool}, we calculate
the invariants of $Y$: $$\chi(\mathcal{O}_Y) = 3, p_g(Y) = h^0(S,
F_1 + 2E_1 + E_3' + E_2') = h^0(S, \phi^*(f_1 + e_1) + E_3' +
E_2') \geq 3,$$ hence $q(Y) \geq 1.$ By Proposition \ref{A.i}, the
Albanese pencil of $Y$ is the pull-back of a pencil $|F|$ of $S$
such that $\pi^*F$ is disconnected for a general element $F$ in
$|F|$. Since $\pi$ is branched over $F_1$, it follows that $F\cdot
F_1 = 0$ and therefore $|F| = |F_1|$. For a general element $F_1
\in |F_1|$, the pull-back $\pi^*F_1$ is an unramified double cover
of $F_1$ given by the relation $2(K_S - (2E_1 + E_3' + E_2') +
E_4)\mid_{F_1}$. Since $\pi^*F_1$ is disconnected, we have $(K_S -
(2E_1 + E_3' + E_2') + E_4)\mid_{F_1} \equiv (K_S -
2E_1)\mid_{F_1} \equiv (K_S - 2E_1 - 2E_3')\mid_{F_1} \equiv (K_S
- F_2)\mid_{F_1}$ is trivial, thus $F_2\mid_{F_1} \equiv K_{F_1}$.
\end{proof}

\begin{Lemma}\label{C.p} Let the notations be as in \ref{not1} and \ref{not2}. Then we have:
\begin{enumerate}
\item[i)]{$\chi(S, K_S + \eta + \eta_i) = 0, ~h^2(S,
K_S + \eta + \eta_i) = 0$;}
\item[ii)]{ $h^0(\mathcal{O}_{F_i}(K_{F_i} + \eta)) \leq 2$;}
\item[iii)]{$h^1(S, \eta - \eta_i) = 1;$}
\item[iv)]{$h^0(S, K_S + \eta_i) = 1, ~h^1(S, K_S + \eta_i) = 0.$}
\end{enumerate}
\end{Lemma}
\begin{proof}
$i)$ By $2\eta \equiv -E_4$, applying Riemann-Roch formula, we get
$\chi(S, K_S + \eta + \eta_i) = 0.$ Using Serre duality, we have
$h^2(S, K_S + \eta + \eta_i) = h^0(S, -\eta + \eta_i)$. Since
$2(-\eta + \eta_i) \equiv E_4$ and $E_4$ is a reduced
$(-4)$-curve, it follows that $h^0(S, -\eta + \eta_i) = 0$.

$ii)$ We show that $h^0(\mathcal{O}_{F_1}(K_{F_1} + \eta)) \leq
2$. Since $\eta_1\mid_{F_1} \equiv \mathcal{O}_{F_1}$, we get the
following exact sequence:
$$0 \rightarrow \mathcal{O}_S(K_S + \eta + \eta_1) \rightarrow
\mathcal{O}_S(K_S + \eta + \eta_1 + F_1) \rightarrow
\mathcal{O}_{F_1}(K_{F_1} + \eta) \rightarrow 0.$$ Considering the
long cohomology sequence, by the results of $i)$, we have
\begin{equation*}
\begin{split}
&h^0(\mathcal{O}_{F_1}(K_{F_1} + \eta))\\
&\leq h^0(\mathcal{O}_S(K_S + \eta + \eta_1 + F_1)) -
h^0(\mathcal{O}_S(K_S + \eta + \eta_1)) + h^1(\mathcal{O}_S(K_S +
\eta + \eta_1))\\
&= h^0(\mathcal{O}_S(K_S + \eta + \eta_1 + F_1)) -
\chi(\mathcal{O}_S(K_S + \eta + \eta_1)) + h^2(\mathcal{O}_S(K_S +
\eta + \eta_1))\\
&= h^0(\mathcal{O}_S(K_S + \eta + \eta_1 + F_1)).
\end{split}
\end{equation*}
Note that $K_S + \eta + \eta_1 + F_1 \equiv 2K_S -
\sum_{i=1}^{i=3}(E_i + E_i') + (E_2 + E_3' - E_3 -E_2') + 2(E_3 +
E_2') \equiv 2K_S - (E_1 + E_1')$. Since the linear system
$|2K_S|$ embeds $E_1 + E_1'$ as a pair of skew lines in
$\mathbb{P}^5$, we obtain $h^0(2K_S - (E_1 + E_1')) = 2$,
thus $h^0(\mathcal{O}_{F_1}(K_{F_1} + \eta)) \leq 2$.

$iii)$ Note that $h^0(S, \eta - \eta_i) = 0$ since $2(\eta -
\eta_i) \equiv -E_4$. Applying Riemann-Roch formula, we get
$\chi(S, \eta - \eta_i)= 1$, thus $-h^1(S, \eta - \eta_i) + h^2(S,
\eta - \eta_i) = 1$. Then applying Serre duality, it suffices to
show $h^0(S, K_S -\eta + \eta_i) = 2$.

We show that $h^0(S, K_S -\eta + \eta_1) = 2$. Since $E_4$ is a
rational $(-4)$-curve and $(2K_S + E_4)E_4 = 0$, we get the
following exact sequence
$$0 \rightarrow \mathcal{O}_S(2K_S) \rightarrow \mathcal{O}_S(2K_S +
E_4) \rightarrow \mathcal{O}_{E_4} \rightarrow 0.$$ From the long
cohomology sequence, using $h^1(S,2K_S) = 0$, we derive $h^0(S,
2K_S + E_4) = 7$. Since $h^0(\Sigma, 3l - e_1 - e_2 - e_3) = 7$
and $2K_S + E_4 \equiv 2(K_S -\eta + \eta_1) \equiv \phi^*(3l -
e_1 - e_2 - e_3)$, we obtain $|2(K_S -\eta + \eta_1)| = \phi^*|3l
- e_1 - e_2 - e_3|$. Notice that $K_S -\eta + \eta_1 \equiv
\sum_{j=1}^{j=3} (E_j + E_j') + (E_2 + E_3' - E_3 - E_2') \equiv
F_1 + E_1' + E_1$, we get $h^0(S, K_S -\eta + \eta_1) = h^0(S, F_1
+ E_1' + E_1) \geq 2$. We may write $|K_S -\eta + \eta_1| = |M| +
F$ where $|M|$ is the moving part and $F$ is the fixed part.
Applying Lemma \ref{B.p}, we can find a divisor $m$ on $\Sigma$
such that $|M| = \phi^*|m|$. Then arguing as in the proof of Lemma \ref{A.prop}, we conclude that $3l - e_1 - e_2 - e_3
- 2m$ is effective. So the only possibility is $m \equiv f_i$ for some $i
\in \{1,2,3\}$, consequently $h^0(S, K_S -\eta + \eta_1) = h^0(S,
M) = h^0(\Sigma, f_i) = 2$.

$iv)$ Since $\eta_i$ is a torsion line bundle of order $2$,
applying Riemann-Roch formula and Serre duality, we get $\chi(S,
K_S + \eta_i) = 1$ and $h^2(S, K_S + \eta_i) = h^0(S, \eta_i) =
0$, thus $h^0(S, K_S + \eta_i) - h^1(S, K_S + \eta_i) = 1$. So it
suffices to show $h^0(S, K_S + \eta_i) \leq 1$. It is a
consequence of Lemma \ref{A.prop} since $2(K_S + \eta_i) \equiv
\phi^*(-K_{\Sigma})$.
\end{proof}
\begin{Corollary}\label{D.p} Let $S$ be as in \ref{not1}.
Let $F_i \in |F_i|,i = 1,2,3$ be a general curve. Then $-\eta +
\eta_j\mid_{F_i} \equiv \mathcal{O}_{F_i}, if~ i \neq j;$
$\eta_i\mid_{F_i} \equiv \mathcal{O}_{F_i}$, $-\eta +
\eta_i\mid_{F_i} \neq \mathcal{O}_{F_i}$. \end{Corollary}
\begin{proof}
By definition, it follows that $\eta_i\mid_{F_i} \equiv
\mathcal{O}_{F_i}$. Using Lemma \ref{A.p}, we have $\eta\mid_{F_1}
\equiv (K_S - (E_1 + E_1'))\mid_{F_1} \equiv K_{F_1} - (E_1 +
E_1')\mid_{F_1} \equiv (F_2 - (E_1 + E_1'))\mid_{F_1} \equiv
(2(E_1 +E_3') - (E_1 + E_1'))\mid_{F_1} \equiv (E_1 -
E_1')\mid_{F_1}$. Notice that $\eta_2\mid_{F_1 } \equiv
\eta_3\mid_{F_1} \equiv (E_1 - E_1')\mid_{F_1}$ and that $\eta_i$
is torsion of order $2$, to prove the corollary, it suffices to
show that $\eta\mid_{F_i} \neq \mathcal{O}_{F_i}$. It is true
because otherwise we have $h^0(\mathcal{O}_{F_i}(K_{F_i} + \eta))=
h^0(\mathcal{O}_{F_i}(K_{F_i})) = 3$ which contradicts Lemma
\ref{C.p} $ii)$.
\end{proof}
\section{Proof of Theorem \ref{Theorem A}}\label{main}

In this section, we will prove the main results. All the
assumptions and notations are as in \ref{not1} and \ref{not2}.
Following the approach in \cite{MP2}, we give 3 involutions on $S$
by considering its fibrations. Let's begin with a lemma.

\begin{Lemma}\label{B.mt} Let $u: S \rightarrow
\mathbb{P}^1$ be a fibration such that $E_4$ is contained in one fiber. Then $u$ is induced by one of the pencils $|F_i|, i=1,2,3$.\end{Lemma}
\begin{proof}
Arguing as in the proof of Lemma 5.7 in \cite{MP2}, we get the lemma.
\end{proof}

\begin{Remark}
Note that in the following proof, Step 2,4,5,6 correspond to Step
1,2,3,4 in the proof of the main result in \cite{MP2}, and the
corresponding argument is nearly the same except that in Step 2. For the readers' convenience, we
give all the details.
\end{Remark}

\begin{proof}[\bf Proof of Theorem \ref{Theorem A}]
Let $\pi_i: Y_i \rightarrow S$ be the double cover branched over
$E_4$ given by the relation $2(-\eta + \eta_i) = E_4$. By
Lemma~\ref{C.p} $iii)$, we have $q(Y_i) = h^1(S, \eta - \eta_i) =
1$. Denote by $\alpha_i: Y_i \rightarrow B_i$ the Albanese pencil
where $B_i$ is an elliptic curve. Corollary \ref{D.p} implies that
$\eta_i \neq \eta_j$ if $i\neq j$, so $\pi_i$ is different from
$\pi_j$. By Proposition~\ref{A.i}, there exists a fibration $h_i: S
\rightarrow \mathbb{P}^1$ and a double cover $\pi_i': B \rightarrow
\mathbb{P}^1$ such that $\pi_i'\circ\alpha_i = h_i\circ\pi_i$. Since
$\pi_i^{-1}(E_4)$ is a rational curve, it must be contained in one
fiber of $\alpha_i$, therefore $E_4$ is contained in one fiber of
$h_i$. By Lemma~\ref{B.mt}, we can find some $s_i \in \{1,2,3\}$
such that $h_i = u_{s_i}$. Then we get the following commutative
diagram:
\[\begin{CD}
Y_i       @>\pi_i>>       S \\
@V\alpha_iVV              @Vu_{s_i}VV \\
B_i        @>\pi_i'>>    \mathbb{P}^1
\end{CD} \]
By Corollary~\ref{D.p}, we have $-\eta + \eta_i\mid_{F_i} \neq
\mathcal{O}_{F_i}$, so a general curve in $\pi_i^*|F_i|$ is
connected, hence $s_i \neq i$. Since $\pi_i': B_i \rightarrow
\mathbb{P}^1$ is branched over four points and $\pi_i: Y_i
\rightarrow S$ is branched over $E_4$, by Proposition~\ref{C.i},
we conclude that the fibration $u_{s_i}: S \rightarrow
\mathbb{P}^1$ has at least three double fibers. As the proof is
long, we break the proof into 6 steps.

Step 1: \emph{The fibration $u_i: S \rightarrow
\mathbb{P}^1,i=1,2,3$ has at most three double fibers.}

We show that $u_3: S \rightarrow \mathbb{P}^1$ has at most three double
fibers. By the analysis above, we can assume $u_{s_3} = u_1$, so
$u_1$ has another double fiber $2N$ aside from $2(E_2 + E_3')$ and
$2(E_3 + E_2')$. Since the curve $\phi(N)$
is irreducible and $NE_1 = 1$, by Proposition~\ref{infofpb}, we
conclude that $N$ is reduced and irreducible. Moreover $\phi$ is
ramified along $N$ since the curve in the pencil $|f_1|$ supported on $\phi(N)$ is reduced. Now to the contrary,
suppose that $u_3$ has two additional double fibers $2M_1, 2M_2$
aside from $2(E_1 + E_2')$ and $2(E_2 + E_1')$. Similarly we can
see that both $M_1$ and $M_2$ are reduced and irreducible, and
$\phi$ is ramified along $M_1$ and $M_2$. Let $R$ be the
ramification divisor of $\phi: S \rightarrow \Sigma$. By the Hurwitz-formula
$K_S \equiv \phi^*(K_{\Sigma}) + R$, we get $R \equiv 3K_S$. Put
$R_0 = \sum_{i=1}^{i=3}(E_i + E_i') + G_1 + G_2 + N + M_1 + M_2$.
Immediately it follows that $R_0 \leq R$ and $2R_0 \equiv \phi^*(8l -
3e_1 - 2e_2 - 3e_3 - 2e_4)$, hence $(R - R_0)G_2 = -1$ and then
$G_2 \leq (R - R_0)$. However, since $\phi : S \rightarrow \Sigma$
is ramified along $G_2$ with branching order 2 and $G_2 \leq R_0$,
$G_2$ cannot be a component of $R - R_0$.

By a similar argument as above or as in the proof of Lemma 5.8 in
\cite{MP2}, one can show that $u_1, u_2$ each has at most $3$ double
fibers.

Step 2: \emph{$(s_1 s_2 s_3)$ is a cyclic permutation.}

Since $s_i \neq i$, it suffices to show that $s_i \neq s_j$ if $i
\neq j$.

We show that $s_1 \neq s_2$. Otherwise we have $s_1 = s_2 = 3$, and
$\alpha_1: Y_1 \rightarrow B_1$ (resp.$\alpha_2: Y_2 \rightarrow
B_2$) arises in the Stein factorization of $u_3 \circ\pi_1$
(resp.$u_3 \circ\pi_2$), i.e., the following commutative diagram
holds.
\[\begin{CD}
Y_1       @>\pi_1>>       S               @<\pi_2<<     Y_2\\
@V\alpha_1VV              @Vu_{3}VV                   @V\alpha_2VV\\
B_1        @>\pi_1'>>    \mathbb{P}^1     @<\pi_2'<<    B_2
\end{CD} \]
Note that for $i =1,2$, $Y_i$ coincides with the normalization of the fiber product
$B_i\times_{\mathbb{P}^1}S$ since $\pi_i$ factors through
the natural projection $B_i\times_{\mathbb{P}^1}S \rightarrow S$
which is also of degree 2, so $\pi_1'$ is different from $\pi_2'$.
Denote by $P_1,P_2,P_3,P_4 = u_3(E_4)$ the branch points of
$\pi_1'$. There exists a branch point $P_5$ of $\pi_2'$ which is
not branched over by $\pi_1'$. Then we find that the fibers over the
points $P_i,i = 1,2,3,5$ of $u_3$ are double fibers. This
contradicts Step 1.

In the following we assume $s_1 = 2, s_2 =3, s_3 =1$. Furthermore we
conclude that for $i=1,2,3$, the fibration $u_i$ has exactly $3$ double fibers.

Step 3: \emph{$\phi^*g_3$ is not reduced.}

With the assumption above, we get the following commutative
diagram:
\[\begin{CD}
Y_2       @>\pi_2>>       S \\
@V\alpha_2VV              @Vu_{3}VV \\
B_2        @>\pi_2'>>    \mathbb{P}^1
\end{CD} \]
Let $W = B_2\times_{\mathbb{P}^1}S$, and denote by $p: W \rightarrow S$
the natural projection which is a double cover. Assume by
contradiction that $G_3 = \phi^*g_3$ is reduced. Since $\pi_2': B_2
\rightarrow \mathbb{P}^1$ is branched over the point $u_3(G_3) =
u_3(E_4)$, the map $p$ is branched over $G_3$, and $W$ is normal
along $p^{-1}G_3$. Since $Y_2$ is the normalization of $W$,  the map
$\pi_2: Y_2 \rightarrow S$ is also branched over $G_3$. This
contradicts the fact that the branch locus of $\pi_2$ is $E_4$.

Step 4: \emph{A general element $F_i \in |F_i|$ is hyperelliptic
for $i = 1,2,3$.}

Let $F_2 \in |F_2|$ be a general fiber of $u_2: S \rightarrow
\mathbb{P}^1$. We show that $F_2$ is hyperelliptic. Since
$\pi_1^*F_2$ (resp.$\pi_1^*F_3$) is disconnected, we may write
$\pi_1^*F_2 = \hat{F_2} + \hat{F_2}'$ (resp.$\pi_1^*F_3 =
\hat{F_3} + \hat{F_3}'$) where the two components are disjoint. By
$F_2F_3 = 4$, we have $\hat{F_2}\hat{F_3} = 2$. Let $p\circ h:
Y_1\rightarrow C\rightarrow \mathbb{P}^1$ be the Stein
factorization of $u_3\circ\pi_1: Y_1 \rightarrow \mathbb{P}^1$.
Since $\hat{F_3}$ is one fiber of $h: Y_1 \rightarrow C$, so
$\hat{F_2}\hat{F_3} = 2$ implies the restriction map
$h\mid_{\hat{F_2}}: \hat{F_2} \rightarrow C$ is a $2$-to-$1$ map.
In addition since $h: Y_1 \rightarrow C$ is not the Albanese map
and $q(Y_1) = 1$, the curve $C$ must be rational. Therefore
$\hat{F_2}$ is hyperelliptic, and so is $F_2$.

Step 5: \emph{$\phi: S \rightarrow \Sigma$ is a Galois cover, and
the Galois group $G \cong \Gamma = Z_2 \times Z_2$.}

For $i=1,2,3$, we denote by $\gamma_i$ the involution on $S$ that induces
the involution on the general fiber $F_i$. Since $S$ is minimal,
the maps $\gamma_i$ are regular maps; and they belong to $\Gamma$ by
Proposition \ref{A.p}. So we only need to show that $\gamma_i \neq
\gamma_j$ if $i \neq j$. Now we show that $\gamma_2 \neq
\gamma_3$. Consider the lifted involution $\hat{\gamma_2}: Y_1
\rightarrow Y_1$. By the construction in Step 4, the restriction
of $\alpha_1$ identifies $\hat{F_3}/\hat{\gamma_2}$ with $B_1$. So
we get $g(\hat{F_3}/\hat{\gamma_2}) = 1$. Then we can see
$\gamma_2 \neq \gamma_3$ since $\hat{F_3}/\hat{\gamma_3} \cong
\mathbb{P}^1$.

Step 6: \emph{$S$ is a Burniat surface.}

The fibration $u_i: S \rightarrow\mathbb{P}^1, i=1,2,3$ has
exactly three double fibers. We denote by $2M_i$ the double fiber
different from $2(E_k + E_j')$ and $2(E_k' + E_j)$ of $u_i$ where
$\{i,j,k\} = \{1,2,3\}$. There exists a fiber $m_i$ of $v_i:
\Sigma \rightarrow \mathbb{P}^1$ such that $2M_i = \phi^*m_i$.
Then $m_i$ is a component of the branch locus. Denote by $D$ be the branch divisor of $\phi$. Then we have
$$-3K_{\Sigma} \equiv D \geq \sum_{i=1}^{i=3}(e_i + e_i' + g_i +
m_i) \equiv -3K_{\Sigma},$$ thus $D = \sum_{i=1}^{i=3}(e_i + e_i'
+ g_i + m_i)$. If we denote by $D_i$ the image of the divisorial part of
the fixed locus of $\gamma_i$, then $D = D_1 + D_2 + D_3$. By step
4, we get $D_1 = g_2 + m_2 + e_1 +e_1'; D_2 = g_3 + m_3 + e_2
+e_2'; D_3 = g_1 + m_1 + e_3 +e_3'$. Finally the theorem is proved.
\end{proof}

\section{Proof of Proposition \ref{mainprop}}\label{prop}
Notations are as in Notation \ref{not1}. We will prove the proposition by contradiction. To the contrary, in this section, by Lemma \ref{A.d}, we assume
$$\phi^*e_i = 2E_i, i
= 1,2,3,4; \phi^*e_j' = 2E_j', \phi^*g_j = 2G_j, j=1,2,3$$ Let $R$
be the ramification divisor. By the formula $K_S \equiv
\phi^*(K_{\Sigma}) + R$, we have $R \equiv 3K_S$, hence $2R \equiv
\phi^*(-3K_{\Sigma})$. Put $R_0 = R - (\sum_{i =1}^{i = 3}(E_i +
E_i' + G_i) + E_4).$ Then $R_0$ is effective, and none of its
irreducible components are in the set $\{E_i, E'_j, G_j\}_{i =
1,2,3,4; j=1,2,3}$. Notice that $2(\sum_{i =1}^{i = 3}(E_i + E_i'
+ G_i) + E_4) \equiv \phi^*(6l - 2e_1 - 2e_2 - 2e_3 - 2e_4)$, we
get
\begin{equation}\label{r0}
2R_0 \equiv \phi^*(3l - e_1 - e_2 - e_3 - e_4) \equiv
\phi^*(-K_{\Sigma}).
\end{equation}
Since $|\phi^*(-K_{\Sigma})| = \phi^*|-K_{\Sigma}|$, we can assume
$2R_0 = \phi^*r_0$ where $r_0$ is an effective divisor on
$\Sigma$.

\begin{Lemma}\label{B.prop}
Let $D \subset S$ be an effective divisor. If $2D \equiv \phi^*(3l
- e_1 - e_2 - e_3 - e_4)$, then every irreducible component of $D$
is contained in $R$.
\end{Lemma}
\begin{proof}
To the contrary, suppose that $D$ contains a reduced and irreducible component
$D_1$ such that $\phi$ is not ramified along it. Denote by $d_1$ be the reduced divisor
supported on $\phi(D_1)$. Thus $ord_{D_1}(\phi^*d_1) =1$. By $2D \in |2K_S| =
\phi^*|-K_{\Sigma}|$, we find a divisor $d \in |-K_{\Sigma}|$ such that $2D = \phi^*d$.
Note that $d_1$ is contained in $d$, and since $ord_{D_1}(\phi^*d) = ord_{D_1}(2D)\geq 2$, we have $d-2d_1$ is effective.
Then we conclude that $d_1$ must be a $(-1)$-curve. So a contradiction follows from the assumptions at the beginning of the section.
\end{proof}

\begin{Lemma}\label{C.prop} The bicanonical map $\phi: S \rightarrow \Sigma$ is ramified along $R_0$
with branching order $2$, and either
\begin{enumerate}
\item[i)]{$R_0$ is irreducible or}
\item[ii)]{$R_0$ can be written as $R_0
= B + F$ where both $B$ and $F$ are irreducible.}
\end{enumerate}
\end{Lemma}
\begin{proof}
Note that by Formula \ref{r0}, we have $R_0 E_i = R_0 E'_j = R_0
G_j = 1$ for $i = 1,2,3,4; j=1,2,3$. To show that $\phi: S
\rightarrow \Sigma$ is ramified along $R_0$ with branching order
$2$, it suffices to show that $R_0$ is reduced. Let $D$ be an
irreducible component of $R_0$. Write $R_0 = mD + R'$ where $m =
ord_{D}(R_0)$. Since the elements in $\{E_i, E'_j, G_j\}_{i =
1,2,3,4; j=1,2,3}$ are integral and span $H^2(S,\mathbb{Q})$,
we can find a curve $C$ in $\{E_i, E'_j, G_j\}_{i = 1,2,3,4;
j=1,2,3}$ such that $DC$ is a non-zero integer. Note that since $C$ is
not a component of $R_0$, we have $DC
> 0$ and $R'C\geq 0$. Then $R_0C=1$ implies that $m = 1$ and $DC =
1$, therefore $R_0$ is reduced. If we denote by $d$ and $c$ the reduced
divisors supported on $\phi(D)$ and $\phi(C)$ respectively, then
we have $d \leq r_0$ and $\phi^*c = 2C$. Write $\phi^*d = 2D +
D'$. By $0 < d\cdot c \leq r_0\cdot c = 1$, we deduce that $d\cdot
c = 1$, and then $\phi^*d\cdot \phi^*c = 4 = (2D)(2C)$. This
implies that $\phi^{-1}(c\cap d) = C\cap D$ and $D'C = 0$. Note that the divisor $D'$, if it is not zero, then it is mapped onto $d_1$, hence it contains the point
$C \cap D$. Then we can see $D' = 0$, because otherwise $D'C>0$.
Therefore we have $\phi^*d = 2D$.

If $R_0$ is reducible, write $R_0 = D_1 + \cdot\cdot\cdot + D_k$
where $D_1, \cdot\cdot\cdot, D_k$ are distinct reduced irreducible
divisors. Denote by $d_i,i=1,2,\cdot\cdot\cdot,k$ the reduced
divisor on $\Sigma$ with support $\phi(D_i)$. Then we have $2D_i =
\phi^*d_i$. Since $d_i$ is not in the set $\{e_i, e'_j, g_j\}_{i =
1,2,3,4; j=1,2,3}$ and $r_0 = \sum_id_i \equiv 3l - e_1 -e_2 -e_3
-e_4$, using the method used in the proof of Lemma~\ref{subdivisor},
we deduce that $k = 2$, and either
\begin{enumerate}
\item[i)]{$d_1 \equiv l, d_2 \equiv 2l - e_1 - e_2 - e_3 - e_4$ or}
\item[ii)]{ $d_1 \equiv l - e_i, d_2 \equiv 2l - e_j - e_k - e_l$
for some $i \in \{1,2,3,4\}$ and $\{i,j,k,l\} = \{1,2,3,4\}$.}
\end{enumerate}
\end{proof}

By the proof of the lemma above, if $R_0$ is reducible, then there
exist two reduced and irreducible divisors $b$ and $f$ on $\Sigma$ such that
$\phi^*b = 2B$, $\phi^*f = 2F$ and $r_0 = b+f$. Recall that in
Section~\ref{construction}, we introduced an automorphism $\tau$
of $\Sigma$ such that $\tau^*l \equiv2l - e_1 - e_2 - e_3$ and $\tau^*(2l - e_1 - e_2 - e_3 - e_4) \equiv l - e_4$. By symmetry, from now on,
if $R_0$ is reducible,
then we assume $b \equiv 2l - e_1 - e_2 - e_3$ and $f \equiv l - e_4$.

\begin{Lemma}\label{kl} With the assumptions above, we have
$$K_S \equiv E_4 + G_1 + G_3 + E_1 + E_2' + E_3
- E_2$$ \end{Lemma}
\begin{proof}
Put $R_1 = E_4 + G_1 + G_3 + E_1 + E_2' + E_3$. Then we get
$$2(R_1 - E_2) \equiv 2K_S.$$ Suppose that the lemma is not true,
then the line bundle $K_S - (R_1 - E_2)$ is torsion of order 2.
Immediately we have $h^0(S, K_S - (R_1 - E_2)) = 0$ and $\chi(S, K_S
- (R_1 - E_2)) = 1$, hence $h^2(S, K_S - (R_1 - E_2)) = h^0(S, R_1 -
E_2) \geq 1$. So there exists an effective divisor $D$ such that $D
\equiv R_1 - E_2$ and a divisor $d$ on $\Sigma$ such that $\phi^*d =
2D$. By Lemma~\ref{B.prop}, every irreducible component of $D$ is
contained in $R$. We will prove that $R_1 - E_2$ cannot be effective
by considering all the possibilities of $D$. Let's begin with a
claim.
\begin{Claim} \label{nojd} $D$ and $R_1$ do not have two common irreducible components $C_1, C_2$
such that $C_1C_2 = 1$.
\end{Claim}
\begin{proof}[Proof of Claim \ref{nojd}]
If there exist two such curves $C_1,C_2$, then we can find $c_1,
c_2 \in \{e_4,g_1,g_3,e_1, e_3,e_2'\}$ such that $\phi^*c_1 =
2C_1$ and $\phi^*c_2 = 2C_2$. Immediately we have $c_1c_2 = 1$ and
$c_1 + c_2 \equiv l-e_i$ for some $i \in \{1,3,4\}$ fixed, thus
$2(R_1 - C_1 - C_2) \equiv 2l - e_j - e_k$ where $\{i,j,k \} =
\{1,3,4\}$. Notice that $-K_{\Sigma} - (2l - e_j - e_k) \equiv l -
e_i - e_2$ is effective, thus applying Lemma~\ref{A.prop}, we obtain
that $h^0(S, R_1 - C_1 - C_2)= 1$. However, this contradicts the
fact that $R_1 - C_1 - C_2$ and $D + E_2 - C_1 - C_2$ are two
different linearly equivalent effective divisors. Then the claim follows.
\end{proof}
By the formula $K_S \equiv \phi^*(K_{\Sigma}) + R$,
we obtain
\begin{equation}\label{KS}
\begin{split}
K_S &\equiv -2(R_1 - E_2) + R \equiv  R - R_1 + E_2  - (R_1 - E_2)\\
&\equiv R_0 + 2E_2 + G_2 + E_1' +E_3' - D
\end{split}
\end{equation}
So $R_0 + 2E_2 + G_2 + E_1' +E_3' - D$ is not effective since
$p_g(S) = 0$.

According to whether $D$ has an irreducible component contained in
$R_0$ or $R_1$, the possibilities for $D$ fall into the two cases:
\begin{enumerate}
\item[Case 1:]{No irreducible component of $D$ is contained in $R_0$, and at least one irreducible component of $D$ is contained in $R_1$.}
\item[Case 2:]{At least one irreducible component of $D$ is contained in $R_0$, or no irreducible component of $D$ is contained in $R_1$.}
\end{enumerate}

Case 1: The following argument is not involved with any
information about the divisor $R_0$. Note that $\phi(R_1)$ contains exactly the six $(-1)$-curves not intersecting $e_2$. So by Fact 3.5 $ii)$, up to an automorphism of
$\Sigma$ fixing $e_2$, we may assume $E_4$ is a common component of $D$ and
$R_1$. Write $D = E_4 + D'$, then $D'E_4 = 2$. Notice that $G_1,
G_2, G_3$ are the only elements in $\{E_i, E'_j, G_j\}_{i =
1,2,3,4; j=1,2,3}$ that intersect $E_4$, so at least one of the
$G_i$'s is a component of $D'$. By Claim \ref{nojd}, neither $G_1$
nor $G_3$ is a component of $D$, so the only possibility is $D' =
2G_2 +E_2' +E_2$, i.e., $D = E_4 + 2G_2 +E_2' + E_2$. We get $E_4
+ G_1 + G_3 + E_1 + E_2' + E_3 \equiv 2G_2 + E_2' + E_4 + 2E_2$,
equivalently $G_1 + G_3 + E_1 + E_3 \equiv 2G_2 + 2E_2$. Then we
have $E_1 + G_1 \mid_{E_1 + G_1} \equiv G_1 + G_3 + E_1 + E_3
\mid_{E_1 + G_1} \equiv 2(E_2 + G_2)\mid_{E_1 + G_1}$ is trivial.
By \cite{BPV} Lemma 8.3 chap.III, this contradicts the fact that $2(G_1 + E_1)$ is a double fiber
of $u_4$.

Case 2: First we list all the possibilities for $D$. We consider the divisor $d$. Note that every
irreducible component of $d$ is either a $(-1)$-curve or is contained in $r_0$.
We break the possibilities into following three cases:
\begin{enumerate}
\item[$\bullet$]{$d = r_0$}
\item[$\bullet$]{$d\neq r_0$, but $d$ and $r_0$ have at least one common component.}
\item[$\bullet$]{$d$ and $r_0$ have no common components}
\end{enumerate}

If we are in the first case, then we get that
\begin{enumerate}
\item[$I:$] {$d = r_0, D = R_0$.}
\end{enumerate}

If $d$ is as in the second case, then $r_0$ is reducible. By assumption, $r_0 = b + f$ where $b$ and $f$ are reduced and irreducible divisors such that $b\equiv 2l -e_1-e_2-e_3$ and $f\equiv l- e_4$. Either $b$ or $f$ is contained in $d$. First if $b \leq d$, then the divisor $d - b \equiv l-e_4$, and its irreducible components are all $(-1)$-curves. The pencil $|l-e_4|$ has exactly 3 reducible elements $g_i + e_i, i = 1,2,3$. So correspondingly, we get the possibilities:
\begin{enumerate}
\item[$II(i):$] {$d = b + g_i + e_i, D = B + G_i + E_i, i = 1,2,3$.}
\end{enumerate}
If $f \leq d$, then similarly the divisor $d- f \equiv 2l -e_1-e_2-e_3$, and its irreducible components are all $(-1)$-curves. Considering the coefficient of $l$ appearing in the equation $d- f \equiv 2l -e_1-e_2-e_3$, we conclude that with multiplicity considered, the divisor $d-f$ contains exactly two irreducible non-exceptional $(-1)$-curves, hence the divisor $d-f$ contains at least one exceptional $(-1)$-curve. So if $e_4$ is contained in $d-f$, then we obtain the possibilities
\begin{enumerate}
\item[$III(i):$] {$d = f + g_i + e_i' + e_4, D = F + G_i + E_i' + E_4, i
= 1,2,3$.}
\end{enumerate}
And for $i =1,2,3$, if $e_i$ is contained in $d-f$, then we get the possibilities
\begin{enumerate}
\item[$IV(i):$] {$d = f + e_i + e_j' + e_k', D = F + E_i +
E_j' + E_k'$ where $\{i,j,k \} = \{1,2,3\}$}.
\end{enumerate}

If $d$ falls into the third case, equivalently, $D$ and $R_0$ have no common components, then the assumptions of Case 2 implies that $D$ and $R_1$ have no common components either. Thus the components of $d$ belong to the set $\{e_2,e_1',e_3',g_2\}$. Then we get following possibilities:
\begin{enumerate}
\item[V:]{$d = g_2 + e_1' + e_3' + 2e_2, D = G_2 + E_1' + E_3' + 2E_2$.}
\end{enumerate}

Since $R_0 + 2E_2 + G_2 + E_1' +E_3' - D$ is not effective, the
cases $I, II(2), IV(2)$ and $V$ cannot happen. And by Claim
\ref{nojd}, we exclude the cases $II(1,3), III(1,3), IV(1,3)$. For the remaining case $III(2)$, i.e., $D = F + G_2 + E_2' + E_4$,
we have $E_4 + G_1 + G_2 + E_1 + E_2' + E_3 - E_2 \equiv F + G_2 +
E_2' + E_4$, equivalently $G_1 + G_3 + E_1 + E_3 \equiv F + G_2 +
E_2$. Restricting the two divisors to the double fiber $2(G_1 +
E_1)$ of $u_4$, by similar argument as in Case 1, we get a
contradiction.

In conclusion the lemma is true.
\end{proof}

\begin{proof}[\bf Proof of Proposition~\ref{mainprop}]
By the same argument as in the proof of Lemma~\ref{kl}, we can
prove
$$K_S \equiv E_4 + G_2 + G_3 + E_2 + E_1' + E_3 - E_1.$$
Consequently we get
$$E_4 +
G_2 + G_3 + E_2 + E_1' + E_3 - E_1 \equiv E_4 + G_1 + G_3 + E_1 +
E_2' + E_3 - E_2,$$ equivalently $G_2 + 2E_2 + E_1' \equiv G_1 +
2E_1 + E_2'$. This implies $h^0(S, G_2 + 2E_2 + E_1') \geq 2$.
Then a contradiction follows from Lemma~\ref{A.prop} since $2(G_2
+ 2E_2 + E_1') \equiv \phi^*(2l - e_4 - e_3)$ and $-K_{\Sigma} -
(2l - e_4 - e_3) \equiv l - e_1 - e_2 \equiv e_3'$. Finally we
finish the proof of the proposition.
\end{proof}

\section{The case when the bicanonical image is singular}\label{bicimage}
\par
In this section, we consider the case when the bicanonical image is singular. First we reduce
the bicanonical image to 6 cases,
then study them case by case.

\subsection{Preparations}\label{pre}
\par

Let $S$ be a minimal surface of general type with $K_S^2 = 5$ and
$p_g = q = 0$; let $\phi: S \rightarrow \Sigma \subset
\mathbb{P}^5$ be the bicanonical map which is a morphism by
\cite{Re}. We assume the degree of $\phi$ is 4, then
$\Sigma$ is a linearly normal surface of degree 5 in
$\mathbb{P}^5$. By \cite{Na},
$\Sigma$ is the image of $\psi: \hat{P} \rightarrow \mathbb{P}^5$
where $\hat{P}$ is the blow-up of $\mathbb{P}^2$ at four points
$P_1, P_2, P_3, P_4$ such that $|-K_{\hat{P}}|$ has no fixed
components and $\psi$ is given by the linear system $|-K_{\hat{P}}|$. The
$P_i$'s can be infinitely near, but it is impossible that two of
them are distinct and both infinitely near to another one. We
denote by $e_i$ the exceptional divisor over $P_i$.
In the previous sections, we have studied the case when $\Sigma$ is smooth, from now on, we assume $\Sigma$ is singular. First we
list
all the the possibilities below.

\begin{itemize}
\item[$\hat{P}_1:$]{$P_1, P_2, P_3$ are distinct and lie on a line, and
$P_4$ is distinct from them;}
\item[$\hat{P}_2:$]{$P_1, P_2, P_3$ lie on a line, $P_3$ is
infinitely near to $P_2$, and $P_1,P_4$ are distinct from them
and each other;}
\item[$\hat{P}_3:$]{$P_1, P_2, P_3$ lie on a line, $P_2$ is infinitely near to $P_1$, $P_3$ is
infinitely near to $P_2$, $P_4$
is distinct from them;}
\item[$\hat{P}_4:$]{$P_1, P_2, P_3$ are distinct and lie on a line,
and $P_4$ is infinitely near to $P_3$;}
\item[$\hat{P}_5:$]{$P_1, P_2, P_3$ lie on a line, $P_3$ is
infinitely near to $P_2$, $P_4$ is infinitely near to $P_3$, $P_1$
is distinct from them;}
\item[$\hat{P}_6:$]{$P_1, P_2, P_3$ lie on a line, $P_2$ is
infinitely near to $P_1$, $P_3$ is infinitely near to $P_2$, and $P_4$
is infinitely near to $P_3$;}
\item[$\hat{P}_3':$]{$P_1, P_2, P_3$ lie on a line, $P_1$ and $P_2$ are distinct, $P_3$ is
infinitely near to $P_2$, $P_4$ is infinitely near to $P_1$;}
\item[$\hat{P}_1':$]{no three points lie on a line, $P_1, P_2, P_3$ are distinct, and $P_4$ is infinitely
near to $P_3$;}
\item[$\hat{P}_2':$]{no three points lie on a line, $P_1$ and $P_3$ are distinct, $P_4$ is infinitely
near to $P_3$, and $P_2$ is infinitely near to $P_1$;}
\item[$\hat{P}_4':$]{no three points lie on a line,  $P_3$ is infinitely near to $P_2$, $P_4$ is infinitely
near to $P_3$, and $P_1$ is
distinct from them;}
\item[$\hat{P}_5':$]{no three points lie on a line, $P_2$ is infinitely
near to $P_1$, $P_3$ is infinitely near to $P_2$, and $P_4$ is
infinitely near to $P_3$.}
\end{itemize}

\begin{Claim} Let the notations be as above. Then for $i = 1,2,3,4,5$, the surface $\hat{P}_i$ is isomorphic to
$\hat{P}_i'$.\end{Claim}
\begin{proof}
We identify $\hat{P}_i$ and $\hat{P}_i'$ by considering their
blowing down map to $\mathbb{P}^2$. For simplicity, we focus on the cases $i=1,2$.

For the surface $\hat{P}_1'$, we denote by $e_1', e_2', e_3'$ the
strict transforms of of the lines through the points $P_2$ and
$P_3$, $P_1$ and $P_3$, $P_1$ and $P_2$ respectively, then they
are all $(-1)$-curves. We get a plane after contracting the curves
$e_1', e_2', e_3', e_4$. Note that $e_3$ is mapped to a line, and $e_1',e_2',e_4$ are mapped to three points on the line. So $\hat{P}_1'$ can be obtained by blowing up of $\mathbb{P}^2$ at 4
points with three of which lying on a line. Then we can identify
$\hat{P}_1'$ with $\hat{P}_1$.

For the surface $\hat{P}_2'$, we denote by $e_3',e_4'$ the strict
transform of the lines through $P_1$ and $P_3$, $P_3$ and $P_4$
respectively. Contracting the curves $e_2, , e_3', e_3, e_4'$, we
get a plane. Similarly we identify $\hat{P}_2'$ with
$\hat{P}_2$.

For the other cases, the proof is similar.
\end{proof}
So in the following, we only need to consider the cases $\hat{P} = \hat{P_i},i=1,2,...,6$.
\begin{Notation}\label{note1} Let $S, \hat{P}, \Sigma$ be as above,
let $\bar{\phi} \circ \eta: S\rightarrow \bar{S} \rightarrow
\Sigma$ be the Stein factorization of $\phi: S \rightarrow
\Sigma$. So $\bar{S}$ is the canonical model of $S$. For a divisor
$D \subset S$ and $e \subset \hat{P}$, denote by $\bar{D}$ the divisor $\eta_*D$ on $\bar{S}$ and by
$\bar{e}$ the divisor $\psi_*e$ on $\Sigma$.
\end{Notation}

\begin{Lemma}\label{key1} Let $e\subset \hat{P}$ be a $(-1)-curve$. Denote by
$E$ the strict transform of $\bar{e}$ with respect to $\phi$.
Then
$K_SE = 2$ and either:
\begin{enumerate}
\item[i)]{ $E$ is reduced and irreducible; or}
\item[ii)]{ $E$ is non-reduced and $E=2E'$; or}
\item[iii)]{ $E = A+B$ where both $A$ and $B$ are irreducible reduced
divisors such that $K_SA= K_SB = 1$.}
\end{enumerate}
Moreover if $E$ is reduced,
then $E^2 \geq -6$.
\end{Lemma}
\begin{proof}
Immediately $K_SE = 2$ follows from $-K_\Sigma \bar{e} = 1$. If $A$ is a reduced and irreducible component of $E$, then $K_SA >0$.
If $E$ is not
reduced, then $E = 2E'$ with $E'$ irreducible and reduced. If $E$
is reducible, then $E=A+B$ where $A$ and $B$ are reduced and irreducible divisors such that $K_SA = K_SB = 1$, moreover we have $A^2 \geq -3$ and $B^2
\geq -3$. And it is easy to see that $E^2 \geq -6$ if $E$ is reduced.
\end{proof}

\begin{Proposition}\label{fibers} Let $\hat{P}$ be as in the precious section, and let
$\hat{g}: \hat{P}\rightarrow \mathbb{P}^1$ be a fibration induced
by the pencil $|f|$ such that every $(-2)-curve$ is contained in
one fiber. Then a general fiber $f \in |f|$ is a
rational curve. Precisely we have:\\
$i)$ if $\hat{P} = \hat{P}_1$, then $|f| = |l-e_i|$~for ~some~$i=1,2,3;$\\
$ii)$ if $\hat{P} = \hat{P}_2$, then $|f| = |l-e_1|;$\\
$iii)$ if $\hat{P} = \hat{P}_4$, then $|f| = |l-e_i|$~for ~some~$i=1,2;$\\
$iv)$ if $\hat{P} = \hat{P}_5$, then $|f| = |l-e_1|;$\\
$v)$ if $\hat{P} = \hat{P}_3$ or~$\hat{P} = \hat{P}_6$, there does
not exist such a fibration.\end{Proposition}
\begin{proof}
For a general element $f\in |f|$, by $-K_{\hat{P}}f>0$, we get
$(K_{\hat{P}} +f)f = 2g(f) - 2<0$, thus it follows that:\\
$1)$ $f$ is a rational curve such that $f^2 = 0$ and $-K_{\hat{P}}f=2$;\\
$2)$ for a $(-2)$-curve $C$, $fC=0$;\\
$3)$ $|f|$ has no fixed part, thus for every effective divisor
$D$, $fD\geq 0$.\\
Note that the Picard group of $\hat{P_i}$ is a free group generated by
$l,~e_1,~e_2,~e_3,~e_4$, so we can set $f \equiv al - b_1e_1 -
b_2e_2 - b_3e_3 - b_4e_4$. By $1)$, we obtain $(a, b_1, b_2,b_3,
b_4) =1$; and by $3)$, we get $a>0, b_i\geq 0$ for $i=1,2,3,4$.

We show the proposition for case $ii)$. In this case, the
equations $f c = 0, f e_2 = 0$ and $-K_{\hat{P}}f= 2$ yield $a -
b_1 - b_3 = 0, ~2b_2 - b_3 = 0$ and $3a - b_1 - b_3 - b_4 = 2$. In
turn we have $a = b_1 +2b_2, ~b_3 = 2b_2$ and $b_4 = 2b_1 + 4b_2 -
2$, thus $f \equiv (b_1+2b_2)l - b_1e_1 - b_2e_2 - 2b_2e_3 -
b_4e_4$.

Denote by $e_4'$ the strict transform of the line through the
point $P_1, P_4$, then $e_4' \equiv l - e_1 -e_4$. By $fe_4'\geq
0$, we obtain
$$a = b_1 + 2b_2 \geq b_1 +b_4 \Rightarrow b_4 \leq
2b_2 \Rightarrow b_4 = 2b_1 + 4b_2 - 2 \leq 2b_2 \Rightarrow b_2
\leq 1 - b_1$$ We have either $b_1 = 1, b_2= b_3 =b_4 = 0$ or $b_1
= 0, b_2 = 1$. The latter is impossible since then $f \equiv 2l
-e_2 -2e_3 -2e_4$, which contradicts $f^2 = 0$.

The other cases can be proven by similar calculations.

\end{proof}

\begin{Remark}\label{rpnf} If the cardinalities of the $(-2)$-curves on $S$
and $\hat{P}$ are equal, then $\bar{S}$ and $\Sigma$ have equal
Picard numbers. If $g: S \rightarrow \mathbb{P}^1$ is a fibration
such that every $(-2)$-curve on $S$ is contained in one fiber, then
it factorizes a fibration of $\bar{g}: \bar{S} \rightarrow
\mathbb{P}^1$. So by Lemma~\ref{pnf}, there exists a
fibration $\bar{u}: \Sigma \rightarrow \mathbb{P}^1$ of $\Sigma$
and a fibration $u: \hat{P} \rightarrow \mathbb{P}^1$ of $\hat{P}$
such that $u = \bar{u} \circ \psi$ and $g = \bar{u} \circ \phi$.
\end{Remark}

\subsection{Analyze all the cases}
\par

All the notations are as in Section \ref{pre}.
And if $P_1,P_2,P_3$ lie on one line, then we denote by $c$ the strict transform of the line.

\subsubsection{The case $\hat{P} = \hat{P}_1$}\label{1}
\par
\begin{Notation} Let $Q = \psi(c)$ be the $A_1-sigularity$ on $\Sigma$.
Denote by $l_i$ the strict transform of a general line through
$P_i$ and by $l_{i4}$ the strict transform of the line through
$P_i$ and $P_4$ for $i = 1,2,3$, then $l_i \equiv l - e_i, ~c
\equiv l - e_1- e_2- e_3$ and $l_{i4} \equiv l - e_i - e_4$.
Denote by $L_i$ (resp. $E_1, E_2, E_3, E_4, L_{i4}$) the strict
transform of $\bar{l_i}$ (resp.$\bar{e_1}, \bar{e_2}, \bar{e_3},
\bar{e_4}, l_{i4}$) with respect to $\phi$.
\end{Notation}

With these
notations, it follows that
\begin{equation}
\begin{split}
-K_{\hat{P}} &\equiv 3l - e_1 - e_2 -e_3 - e_4 \\
&\equiv 2l_4 + c + e_4\\
&\equiv l_{i4} + 2e_i + l_i + c
\end{split}
\end{equation}
consequently
\begin{equation}
\begin{split}
-K_{\Sigma}
&\equiv 2\bar{l_4} + \bar{e_4}\\
&\equiv \bar{l_{i4}} + 2\bar{e_i} + \bar{l_i}
\end{split}
\end{equation}
Note that there exists at most one $(-2)$-curve on $S$ since
$\rho(S) - \rho(\Sigma) = 1$.

\begin{Lemma}\label{reducedornot} $E_4 = 2E_4'$ for some divisor $E_4'$ on $S$.
\end{Lemma}
\begin{proof}
By $-K_{\Sigma} \equiv 2\bar{l_4} + \bar{e_4}$, we may write
$$2K_S \equiv 2L_4 + E_4 + Z$$ where $Z$ is zero or supported on a
$(-2)$-curve. We write $Z = 2Z' + Z''$ where $Z''$ is reduced.

Suppose otherwise that $E_4$ is reduced. We have $2(K_S -
L_4 - Z') \equiv E_4 + Z''$, and it gives a double
cover $\pi: Y \rightarrow S$ branched along $E_4$ and $Z''$. Observing that $E_4 + Z'' = \phi^*e_4 + Z_1$ where $Z_1$ is zero or supported on some $(-2)$-curves, we conclude that $(E_4 + Z'')^2 = (\phi^*e_4)^2 + Z_1^2 \leq -4$ and equality holds if and only if $E_4 + Z'' = \phi^*e_4$. Note that if $\phi^*e_4$ contains some $(-2)$-curves, i.e, $E_4\neq \phi^*e_4$, then $E_4^2 < -4$, hence $E_4^2 = -6$ by Lemma \ref{key1}. Using Formula \ref{tool} and Lemma \ref{pullback}, we calculate the invariants of $Y$:
\begin{equation}
\begin{split}
\chi(Y) &= 2 + \frac{1+ \frac{(E_4 + Z'')^2}{4}}{2} \leq 2\\
p_g(Y) &= h^0(S, 2K_S - L_4 - Z')\\
&= h^0(S, \phi^*(\bar{l_4} + \bar{e_4}) \\
&\geq h^0(\mathcal{O}_\Sigma(\bar{l_4} + \bar{e_4})) \geq h^0(\mathcal{O}_{\hat{P}}(\psi^*(\bar{l_4} + \bar{e_4})))\geq h^0(\hat{P}, l_4 + e_4) =  3\\
q(Y) &\geq 2
\end{split}
\end{equation}
Note that $(E_4 + Z'')^2 \geq E_4^2 + Z''^2 \geq -8$ since $Z''^2 \geq -2$. To guarantee that $\chi(Y)$ is an integer, we have $(E_4 + Z'')^2 = -4$ and thus $E_4 + Z'' = \phi^*e_4$. In turn we get that either
\begin{enumerate}
\item[i)]{$\phi^*e_4 = E_4$ is a smooth rational $(-4)$-curve; or}
\item[ii)]{ $E_4^2 = -6, Z''^2 = -2, E_4Z'' = 2$, thus $E_4 = A+B$ where $A,B$ are smooth rational $(-3)$-curves and $Z''$ is composed with a $(-2)$-curve.}
\end{enumerate}
In any case, $E_4+ Z''$ has at most simple [3]-points or double points which are negligible singularities. So $Y$ has at most canonical singularities. Then we can apply Proposition~\ref{B.i} and get a contradiction since $K_Y^2 = 2(K_S +
\frac{E_4 + Z''}{2})^2 \leq 12 < 16(q(Y) - 1)$. Hence $E_4$ is reduced, and we can write $E_4 = 2E_4'$.
\end{proof}

\begin{itemize}
\item[Case 1:] {There's a $(-2)$-curve $\theta$ on $S$.}
\end{itemize}

Step 1: Obtain a double cover and a fibration of $S$.

As in the proof of the lemma above, we have $2K_S \equiv 2L_4 + E_4 +
Z$ and write $Z = 2Z' + Z''$ where $Z''$ is reduced, then get a relation
$$2(K_S - L_4 - E_4' - Z') = Z''$$
which gives a double cover $\pi: Y \rightarrow
S$. Using Formula \ref{tool}, calculate the invariants of $Y$:
\begin{equation}
\begin{split}
\chi(Y) &= 2 + \frac{\frac{Z''^2}{4}}{2} \leq 2\\
p_g(Y) &= h^0(S, 2K_S - L_4 -E_4' - Z') = h^0(S,  L_4 + E_4' + Z' + Z'') \geq 2\\
q(Y)&\geq 1.
\end{split}
\end{equation}
Remark that $Z''= 0$ to guarantee that $\chi(Y)$ is an integer. Then applying Proposition~\ref{A.i},
we get a commutative diagram as follows:
\[\begin{CD}
Y          @>\pi>>          S \\
@V\alpha VV                @VgVV \\
B        @>\pi'>>      \mathbb{P}^1
\end{CD} \]
where $\alpha: Y \rightarrow B$ is the Albanese pencil and $\pi':B
\rightarrow \mathbb{P}^1$ is a double cover. Since $\pi^*\theta$
is composed with 2 $(-2)$-curves, $\theta$ is contained in one
fiber of $g: S \rightarrow \mathbb{P}^1$. Then Remark~\ref{rpnf} gives a fibration $v: \Sigma \rightarrow \mathbb{P}^1$ such that $g
= v\circ \phi$. By Proposition~\ref{fibers}, we may assume that $g$ is induced by the pencil $|L_1| = \phi^*\bar{|l_1|}$.

Step 2: Analyze the ramification divisor of $\bar{\phi}$.

Now consider the ramification divisor $R$ of the map $\bar{\phi}:
\bar{S}\rightarrow \Sigma$. We have $R \equiv 3K_{\bar{S}}$. Since $\pi':B \rightarrow
\mathbb{P}^1$ is branched over 4 points, $g: S \rightarrow
\mathbb{P}^1$ has at least 4 double fibers, we select 4 and denote them by
$2M_1, 2M_2, 2M_3, 2M_4$. Note that since every curve in $|\bar{l_1}|$
is reduced, $\bar{\phi}$ is ramified along $\bar{M_i},i=1,2,3,4$, thus $\bar{M_1} + \bar{M_2} + \bar{M_3} + \bar{M_4} \leq R$. Then we have $(R - (\bar{M_1} + \bar{M_2} + \bar{M_3} +
\bar{M_4}))(\bar{\phi}^*e_1) = -2 = (\bar{\phi}^*e_1)^2$ which
contradicts Lemma~\ref{ramification}. Therefore this case does not occur.

\begin{Remark}
The process in Step 2 above will be frequently used in the following proof,
and we call it $ARDP$, namely analyzing the ramification divisor
process. \end{Remark}

Case 2: There's no $(-2)$-curve on $S$.

In this case, $E_i =
\phi^*\bar{e_i}$ and $L_1 = \phi^*\bar{l_1}$ are all Cartier
divisors. By $E_i^2 = -2$ for $i = 1,2,3$, applying Lemma \ref{key1}, we conclude that the
$E_i$'s are reduced.

\begin{Claim} $\phi^*l_{i4} = L_{i4} = 2L_{i4}'$ is non-reduced for $i = 1,2,3$.\end{Claim}
\begin{proof}
To the contrary, suppose that $L_{i4}$ is reduced. By $K_\Sigma \equiv \bar{l_{i4}}
+ 2\bar{e_i} + \bar{l_i}$, we have $2K_S \equiv L_{i4} + 2E_i +
L_i$, thus $$2(K_S - E_1 ) \equiv L_i + L_{i4}.$$ Then by the
relation above, we get a double $\pi: Y \rightarrow S$.
Calculating the invariants of $Y$, we obtain
$$\chi(Y) = 3,~~p_g(Y) = 4,~~q(Y)= 2.$$ By
Proposition~\ref{A.i}, we get the following commutative diagram:
\[\begin{CD}
Y          @>\pi>>          S \\
@V\alpha VV                @VgVV \\
B        @>\pi'>>      \mathbb{P}^1
\end{CD} \]
where $\alpha: Y \rightarrow B$ is the Albanese pencil and $\pi':B
\rightarrow \mathbb{P}^1$ is a double cover. Since $L_i^2 = 0$,
$g: S \rightarrow \mathbb{P}^1$ must be induced by the
pencil $|L_i|$. The map $\pi': B \rightarrow \mathbb{P}^1$ is
branched along 6 points, so there exist at least 4 double fibers, and we choose four double fibers: $2M_1,2M_2, 2M_3, 2M_4$. Then considering the
ramification divisor of the map $\bar{\phi}: \bar{S}\rightarrow
\Sigma$ and going process $ARDP$, we get a
contradiction.
\end{proof}

By the Claim above, we get a relation $2(K_S - L_{i4}' - E_i)
\equiv L_i$, and then a double $\pi: Y \rightarrow S$.
Calculating the invariants of $Y$, we obtain
$$\chi(Y) = 3,~~p_g(Y) = 3,~~q(Y)= 1.$$ By
Proposition~\ref{A.i}, we get the following commutative diagram:
\[\begin{CD}
Y          @>\pi>>          S \\
@V\alpha VV                @VgVV \\
B        @>\pi'>>      \mathbb{P}^1
\end{CD} \]
where $\alpha: Y \rightarrow B$ is the Albanese pencil and $\pi':B
\rightarrow \mathbb{P}^1$ is a double cover. Similarly we can see
$g: S \rightarrow \mathbb{P}^1$ is induced by the
pencil $|L_i|$. The map $\pi': B \rightarrow \mathbb{P}^1$ is
branched at 4 points, so there exist at least 3 double fibers. Select two double fibers
$2M_{i1}, 2M_{i2}$ different from
$2(L_{i4}' + E_4')$. For $i_1 \neq i_2$, if $M_{i_1*}$ and $M_{i_2*}$ have common components, then the image of the common components under the bicanonical map must be contained in some
fibers of the pencils $|\bar{l_{i_1}}|$ and $|\bar{l_{i_2}}|$. Note that the fiber containing $E_i,i=1,2,3$ must be reduced since $\phi^*e_i = E_i$ is reduced in this case. So the image of $M_{ij}$ does not contain
$\bar{e_i}, i =1,2,3,4$, as an easy consequence, any two of the $M_{ij}$'s have no
common components, thus $R_1 = \sum_{i = 1}^{i=3}(M_{i1} + M_{i2}
+ L_{i4}') + E_4' \leq R$.
Put $R_1 = \sum_{i = 1}^{i=3}(M_{i1} + M_{i2}+ L_{i4}') + E_4'$.
It follows that $2R_1 \equiv \phi^*(9\bar{l} - 3\bar{e_1} - 3\bar{e_2} - 3\bar{e_3} - 2\bar{e_4})$ and $R_1 \leq R$. Notice that $2R \equiv 6K_S
\equiv \phi^*(9\bar{l} - 3\bar{e_1} - 3\bar{e_2} - 3\bar{e_3} - 3\bar{e_4})$ which
contradicts $R_1 \leq R$, so we are done.

In conclusion, we have $\hat{P} \neq \hat{P}_1$.

\subsubsection{The case $\hat{P} = \hat{P}_2$}\label{2}
\par

We fail to exclude this case, but we can describe the fibration of
$S$ and give an effective divisor linearly equivalent to $K_S$.

\begin{Notation} Let $Q = \psi(c)$ be the $A_1$-sigularity on
$\Sigma$. Denote by $l_1$ the strict transform of a general line
through $P_1$, by $\hat{e_1}$ the strict transform of the line
through $P_4$ and $P_2$, by $\hat{e_4}$ the strict transform of
the line through $P_4$ and $P_1$. Then $l_1 \equiv l - e_1, c
\equiv l - e_1- e_2- 2e_3, \hat{e_1} \equiv l - e_2 - e_3 - e_4,
\hat{e_4} \equiv l - e_1 - e_4$. Denote by $L_1$ (resp. $\hat{E_1},
\hat{E_4}, E_3, E_4$) the strict transform of
$\bar{l_1}$ (resp.$\bar{\hat{e_1}}, \bar{\hat{e_4}}, \bar{e_3},
\bar{e_4}$) w.r.t. $\phi$.
\end{Notation}

Immediately we have
\begin{equation}
\begin{split}
-K_{\hat{P}} &\equiv 3l - e_1 - e_2 - 2e_3 - e_4 \\
&\equiv 2l_4 + c + e_4\\
&\equiv 3l_1 + 2e_1 - e_2 - 2e_3 - e_4\\
\end{split}
\end{equation}
and then
\begin{equation}
\begin{split}
-K_{\Sigma}
&\equiv 2\bar{l_4} + \bar{e_4}\\
&\equiv 3\bar{l_1} + 2\bar{e_1} - 2\bar{e_3} - \bar{e_4}\\
\end{split}
\end{equation}
Note that there exist at most two $(-2)$-curves on $S$ since
$\rho(S) - \rho(\Sigma) = 2$.

\begin{Lemma}Let the notations be as above. Then $\phi^*\bar{e_4} = 2E_4'$ for some divisor $E_4'$ on $S$.\end{Lemma}
\begin{proof}
By $-K_{\Sigma} \equiv 2\bar{l_4} + \bar{e_4}$, we may write
$$2K_S \equiv 2L_4 + E_4 + Z$$ where $Z$ is zero or supported on a
$(-2)$-curve. Note that $Z$ arises from $\phi^*\bar{e_4}$ or
$\phi^*2\bar{l_4}$. We write $Z = 2Z' + Z''$ where $Z''$ is reduced.

Since $Z''$ contains at most two $(-2)$-curves, so $Z''^2 \geq -4$. Arguing as in the proof of Section~\ref{1}
Lemma~\ref{reducedornot}, we get that $E_4$ is reduced.

We write $E_4 = 2D_4$ and then write $\phi^*e_4 = 2D_4 + 2Z_1' + Z_1''$. We still need to prove $Z_1''=0$. Considering the relation $2(K_S - L_4 - D_4 - Z') = Z''$,
since $K_S(K_S - L_4 - D_4 - Z') = 0$ is even, $(K_S - L_4 - D_4 - Z')^2$ is even, hence $8|Z''^2$. It is only possible that $Z''=0$. Note that since $L_4$ moves, we can assume $Z_1''$ is not contained in $\phi^*\bar{l_4}$. So $Z''=0$ implies $Z_1''=0$, then we are done.
\end{proof}

\begin{Lemma}\label{numbofdoublefibers}Let $g: S \rightarrow \mathbb{P}^1$ be the fibration
induced by the pencil $|L_1|$. Then $g$ has at most 4 double
fibers.\end{Lemma}
\begin{proof} Otherwise $\bar{\phi}$ is ramified along at least 4 double fibers
since the fibration $u: \Sigma \rightarrow \mathbb{P}^1$ induced
by $|\bar{f_1}|$ has unique non-reduced fiber $2\bar{e_3}$.
Considering the ramification divisor of the map $\bar{\phi}: \bar{S}
\rightarrow \Sigma$ and going process ARDP, we get a contradiction.
\end{proof}

By $-K_\Sigma \equiv 3\bar{l_1} + 2\bar{e_1} - 2\bar{e_3}  -
\bar{e_4}$, we have $2K_S \equiv 3L_1 + 2E_1 - 2E_3 - 2E_4' + Z$
where $Z = \phi^*(2\bar{e_1} - 2\phi^*\bar{e_3}) - 2E_1 + 2E_3$.
Write $Z = 2Z' + Z''$, we get the relation $2(K_S - L_1 - E_1 +
E_3 + E_4' - Z') \equiv Z'' + L_1$, and then a double cover $\pi:
Y \rightarrow S$ branched along $L_1$ and $Z''$. Note that $L_1Z''=0$ and $L_1^2 = 0$. By Formula \ref{tool} and Lemma \ref{pullback}, we calculate the invariants of $Y$:
\begin{equation}
\begin{split}
\chi(Y) &= 2 + \frac{2+ \frac{(L_1 + Z'')^2}{4}}{2} = 2+ \frac{2+ \frac{Z''^2}{4}}{2}\\
p_g(Y) &= h^0(S, 2K_S - L_1 - E_1 +
E_3 + E_4' - Z')\\
&\geq h^0(S, 2L_1 + E_1 -
E_3 - E_4' - Z') \geq h^0(\mathcal{O}_\Sigma(2\bar{l_1} + \bar{e_1} - \bar{e_3} - \bar{e_4}))\\
&\geq h^0(\mathcal{O}_{\hat{P}}(\psi^*(2\bar{l_1} + \bar{e_1} - \bar{e_3} - \bar{e_4}))) \geq h^0(\mathcal{O}_{\hat{P}}(2l - e_1-e_2-e_3 - e_4)) = 3
\end{split}
\end{equation}
and deduce that $Z'' = 0$ to guarantee that $\chi(Y)$ be an integer, thus $\chi(Y) = 3$ and $q(Y) \geq 1$, moreover we have
\begin{equation}\label{equ23}
\phi^*(\bar{e_1} - \bar{e_3}) = E_1 - E_3 + Z' \end{equation}
By
Proposition~\ref{A.i}, we get a commutative diagram
\[\begin{CD}
Y          @>\pi>>          S \\
@V\alpha VV                @VgVV \\
B        @>\pi'>>      \mathbb{P}^1
\end{CD} \]
where $\alpha: Y \rightarrow B$ is the Albanese pencil and $\pi':B
\rightarrow \mathbb{P}^1$ is a double cover. By Lemma
\ref{numbofdoublefibers}, we get that $g(B) = 1$ because otherwise
$g$ will has at least 5 double fibers. Denote by $P_i,i=1,2,3,4$ the branch points of $\pi'$, and assume $g^*P_i = 2M_i,i=1,2,3$ and $g^*P_4 = L_1$. Since $Y$ can be obtained
by the normalization of fiber product $S\times _{\mathbb{P}^1}B$, by Lemma \ref{cvroffbr}, we have $K_S - L_1 - E_1 + E_3 + E_4' - Z'|_F = 0$ if $F$ is a fiber of $g$ different from $2M_i$, and $K_S - L_1 - E_1 + E_3 + E_4' - Z'|_{M_i} = M_i|_{M_i}$. Considering the Chern classes, we conclude that
\begin{equation}\label{equ21}
K_S - L_1 - E_1 + E_3 + E_4' - Z' \equiv M_1 +M_2 - M_3
\end{equation} where the $2M_i$'s are double fibers of $g$.

Similarly, we get another relation
\begin{equation}\label{equ22}
K_S - L_1 - \hat{E_1} + E_3 + \hat{E_4'} - \hat{Z'} \equiv
\hat{M_1} + \hat{M_2} - \hat{M_3}\end{equation} where the
$2\hat{M_i}$'s are double fibers of $g$ and
\begin{equation}\label{equ24}
\phi^*(\bar{\hat{e_1}} - \bar{e_3}) = \hat{E_1} - E_3 + \hat{Z'}
\end{equation}
Since $g$ has at most 4 double fibers, we may assume that
$\hat{M_i} = M_i$ for $i = 1, 2$. It follows that
$$M_1 + M_2 \equiv K_S - L_1 - E_1 + E_3 + E_4' - Z' + M_3 \equiv
K_S - L_1 - \hat{E_1} + E_3 + \hat{E_4'} - \hat{Z'} + \hat{M_3},$$
thus
$$- E_1  + E_4' - Z' + M_3 \equiv - \hat{E_1} +
\hat{E_4'} - \hat{Z'} + \hat{M_3}.$$ The equation \ref{equ23} minus
\ref{equ24} yields $\phi^*\bar{\hat{e_1}} - \phi^*\bar{e_1} = \hat{E_1} + \hat{Z'} - (E_1 - Z')$, thus $\hat{M_3} + \phi^*\bar{e_1} + \hat{E_4'}
\equiv M_3 + \phi^*\bar{\hat{e_1}} + E_4'$. Since $2(\hat{M_3} +
\phi^*\bar{e_1} + \hat{E_4'}) \equiv 2K_S$ and $|2K_S| =
\phi^*|-K_\Sigma|$, $\hat{M_3} + \hat{E_4'} + E_4' + M_3 +
\phi^*(\bar{\hat{e_1}}+\bar{e_1})$ must be the pull-back of some
element in $|-K_\Sigma|$. Since $2\hat{M_3}, 2(\hat{E_4'} + E_4'),
2M_3$ are 3 double fibers (maybe the same), $\phi$ is ramified
along them unless one of them is
$\phi^*\bar{e_3}$. We assume $M_3 = \phi^*\bar{e_3}$, then the only
possibility is $\hat{M_3} = \hat{E_4'} + E_4'$. By \ref{equ21}, we
have $$K_S\equiv \phi^*\bar{e_1} - E_4' + M_1 + M_2.$$

\begin{Claim} $S$ has at most one $(-2)$-curve, and
the bicanonical map $\phi: S \rightarrow \Sigma$ can not lift to a morphism to $\hat{P}$.
\end{Claim}
\begin{proof}
Consider the double cover $\pi: X \rightarrow S$ given by the relation $2(M_1 - M_2) \equiv \mathcal{O}_S$. By $K_S\equiv \phi^*\bar{e_1} - E_4' + M_1 + M_2$, using Formula \ref{tool} and Lemma \ref{pullback}, we get the invariants of $X$ are as follows:
\begin{equation}
\begin{split}
\chi(X) &= 2\\
p_g(X) &= h^0(S, \phi^*\bar{e_1} - E_4' + M_1 + M_2 + (M_1 - M_2)) = h^0(S, \phi^*\bar{e_1} - E_4' + L_1)\\
&\geq h^0(\mathcal{O}_\Sigma(\bar{l_1} + \bar{e_1} - \bar{e_4})\geq h^0(\mathcal{O}_{\hat{P}}(l-e_4)) = 2\\
q(X) &\geq 1
\end{split}
\end{equation}
Therefore we get a commutative diagram as follows:
\[\begin{CD}
X         @>\pi>>          S \\
@V\alpha VV                @VfVV \\
B        @>\pi'>>      \mathbb{P}^1
\end{CD} \]
where $\alpha: Y \rightarrow B$ is the Albanese pencil and $\pi':B
\rightarrow \mathbb{P}^1$ is a double cover. We claim that the fibration $f:S\rightarrow \mathbb{P}^1$
is different from the fibration $g: S\rightarrow \mathbb{P}^1$  defined by the pencil $|L_1|$. Indeed, considering the Stein factorization $\pi'' \circ h: X \rightarrow C \rightarrow \mathbb{P}^1$ of the map $g \circ \pi$, $\pi''$ is a double cover branched along the two points $g(M_1),g(M_2)$, so $f$ is different from $g$. If $S$ has two $(-2)$ curves, then applying Remark \ref{rpnf}, Proposition \ref{fibers} tells that $f$ coincides with $g$, thus a contradiction follows.

If $\phi$ lifts to a morphism $\hat{\phi}: S \rightarrow \hat{P}$, then $\hat{\phi}^*cK_S = \hat{\phi}^*e_2K_S = 0$, thus $\phi$ contracts at least two $(-2)$-curves. For the same reason, this is impossible.
\end{proof}

In conclusion, with the notations and assumptions above, we have
\begin{enumerate}
\item[i)]{The pencil $|L_1|$ induces a genus 3 fibration with exactly 4 double fibers: $M_1, M_2, M_3=\phi^*\bar{e_3}, \hat{M_3} = 2(\hat{E_4}' + E_4')$;}
\item[ii)]{$K_S\equiv \phi^*\bar{e_1} - E_4' + M_1 + M_2$;}
\item[iii)]{$S$ contains at most one $(-2)$-curve, and the bicanonical map $\phi: S \rightarrow \Sigma$ can not lift to a morphism to $\hat{P}$.}
\end{enumerate}

\subsubsection{The case $\hat{P} = \hat{P}_4$}\label{4}

\begin{Notation} Denote by $l_3$ the strict transform of a general line
through $P_3$ and by $l_1$ the strict transform of a general line
through $P_1$. Let $Q = \psi(c + e_3)$ be the $A_2$-singularity
on $\Sigma$. Denote by $L_3$ (resp.$L_1, E_1, E_4, E_2$) the strict
transform of $\bar{l_3}$ (resp.$\bar{l_1}, \bar{e_1}, \bar{e_4},
\bar{e_2}$) with respect to $\phi$. \end{Notation}

With the notations above, it follows that
$$-K_{\hat{P}} \equiv 3l - e_1 - e_2 -e_3 -2e_4 \equiv 2l_3 + e_4
+ c + 2e_3$$ thus
$$-K_{\Sigma} \equiv 2\bar{l_3} + \bar{e_4}$$ Immediately we have
$$\psi^*\bar{l_3} = l_3 + \frac{2}{3}e_3+ \frac{1}{3}c,~~~~~~~~~
\psi^*\bar{e_4} = e_4 + \frac{2}{3}e_3+ \frac{1}{3}c$$
$$\psi^*\bar{e_2} = e_2 + \frac{2}{3}c + \frac{1}{3}e_3,~~~~~~~~~
\psi^*\bar{e_1} = e_1 + \frac{2}{3}c + \frac{1}{3}e_3$$ and then
$$\bar{l_3}\bar{e_1} = \bar{e_4}\bar{e_1} = \frac{1}{3},~~~
\bar{e_2}\bar{e_1} = \frac{2}{3},~~~\bar{e_1}^2 = -\frac{1}{3}.$$

Note that there exist at most two $(-2)$-curves on $S$ since
$\rho(S) - \rho(\Sigma) = 2$. First we introduce the following
claim.

\begin{Claim}\label{layoutof(-2)-curves} There are exactly two $(-2)$-curves which we denote by $\theta_1,\theta_2$
such that $\theta_1\theta_2 = 1$ and $\phi^{-1}Q = \theta_1\cup
\theta_2$.
\end{Claim}
\begin{proof}
By $\bar{e_1}^2 = -\frac{1}{3}$, we have $(\phi^*\bar{e_1})^2 =
-\frac{4}{3}$. Write $\phi^*\bar{e_1} = E_1 + x\theta_1 +
y\theta_2$ where $\theta_i$ is a $(-2)$-curve or zero and $x,y$
are rational numbers. Note that if $\theta_i$ is nonzero then so is its coefficient. We set the coefficient of $\theta_i$ to be
zero if $\theta_i$ is zero, so it makes sense to set
$\theta_i^2 = -2$. Then we have
$$(\phi^*\bar{e_1})^2 = E_1^2 + 2xE_1\theta_1 + 2yE_1\theta_2 -2x^2 +
2xy\theta_1\theta_2 - 2y^2 = -\frac{4}{3},$$ therefore one of the denominators of $x,y$ is divided by 3. By $\phi^*\bar{e_1}\theta_1 = \phi^*\bar{e_1}\theta_2 = 0$,
we get
\begin{equation}
\begin{split}
2x - \theta_1\theta_2y = E_1\theta_1\\
-\theta_1\theta_2x - 2y = E_2\theta_2
\end{split}
\end{equation}
Applying Crammar's rule, we get that $3|4 - \theta_1\theta_2$,
hence $\theta_1\theta_2 = 1$. So the claim is true.
\end{proof}
Denote by $\theta_1, \theta_2$ the two $(-2)$-curves on $S$. By
$-K_{\Sigma} \equiv 2\bar{l_3} + \bar{e_4}$, we may write $$2K_S =
2L_3 + E_4 + Z,$$ then $L_3Z>0, E_4Z>0$ since $\bar{l_3}$ and
$\bar{e_4}$ pass through the singular point on $\Sigma$. Set $Z =
a\theta_1 +b\theta_2$ where $a,b$ are positive integers. We assume
$b\leq a$. By
$$K_SL_3 = 4, K_SE_4 = 2, K_S\theta_i = 0~i=1,2,$$ we get that
\begin{equation}\label{inequt1}
8 = 2K_SL_3 = 2L_3^2 + L_3E_4 + L_3Z \Rightarrow L_3Z = 8 - 2L_3^2
- L_3E_4 >0
\end{equation}
\begin{equation}\label{inequt2}
4 = 2K_SE_4 = 2L_3E_4 + E_4^2 + E_4Z \Rightarrow E_4Z = 4 - E_4^2
- 2L_3E_4 >0
\end{equation}
\begin{equation}\label{equt1}
\begin{split}
&0 = 2K_SZ = 2L_3Z + E_4Z + Z^2 \\
&\Rightarrow -Z^2 = 2L_3Z + E_4Z = 20 - 4L_3^2 - 4L_3E_4 - E_4^2\\
& \Rightarrow a^2 + b^2 - ab = 10 - 2L_3^2 - 2L_3E_4 -
\frac{E_4^2}{2}
\end{split}
\end{equation}
By $E_4\theta_i \geq 0, L_3\theta_i\geq 0$, we have $Z\theta_i
\leq 0$ which implies
\begin{equation}\label{inequt3}
b\leq a \leq 2b
\end{equation}

Applying the argument in Section~\ref{1} Case 1 to the relation $2K_S =
2L_3 + E_4 + Z$, we exclude the case when
$E_4$ is non-reduced. From now on, we assume that $E_4$ is reduced.

By $K_SL_3 = 4$, applying Hodge index theorem, we have $L_3^2 =
0$~or~$2$. By Lemma \ref{key1}, we have $E_4^2 = -2$ or~
$-4$~or~$-6$. By $L_3Z = 8 - 2L_3^2 - L_3E_4 \leq 8$, we have
\begin{equation}\label{inequt4}
min\{a,b\}\leq 8
\end{equation}

With the equation \ref{equt1} and the inequalities \ref{inequt1},
\ref{inequt2}, \ref{inequt3}, \ref{inequt4}, going a computer
program, we list all the possibilities:
\begin{align*}
&a  &  &b  &  &L_3^2    &  &L_3E_4    &  &E_4^2    &  &E_4Z \\
&2  &  &1  &  &2        &  &2         &  &-2       &  &2      \\
&2  &  &2  &  &2        &  &2         &  &-4       &  &4      \\
&3  &  &2  &  &0        &  &2         &  &-2       &  &2      \\
&3  &  &2  &  &2        &  &0         &  &-2       &  &6      \\
&3  &  &3  &  &0        &  &1         &  &-2       &  &4      \\
&4  &  &2  &  &0        &  &0         &  &-4       &  &8      \\
&2  &  &1  &  &2        &  &3         &  &-6       &  &4      \\
&3  &  &2  &  &0        &  &3         &  &-6       &  &4      \\
&3  &  &2  &  &2        &  &1         &  &-6       &  &8      \\
&3  &  &3  &  &0        &  &2         &  &-6       &  &6      \\
&3  &  &3  &  &2        &  &0         &  &-6       &  &10      \\
&4  &  &3  &  &0        &  &0         &  &-6       &  &10     \\
\end{align*}

Write $Z = 2Z' + Z''$ where $Z''$ is reduced. Then by $2(K_S - L_3
- Z') \equiv E_4 + Z''$, we get a double cover $\pi:Y \rightarrow
S$ branched over $E_4$ and $Z''$. Check case by case listed in
the table that $E_4 + Z''$ has at most negligible singularities as follows: say for the case in the first row, we have $Z''= \theta_2$, and by $\theta_2Z = 0$, we deduce $E_4Z'' = 0$, then since $E_4$ has arithmetic genus 1, we can see that $E_4 + Z''$ has at most negligible singularities; say for the case in the 9th row, we have $Z''=\theta_1$ and $E_4\theta_1 \leq 2$, then since $E_4$ is composed with 2 disjoint smooth $(-3)$-curves, $E_4\cap Z''$ are the singularities of $E_4 + Z''$, thus $E_4 + Z''$ has at most double points, so we are done. So $Y$ has at most
canonical singularities. Then we can use Formula \ref{tool} to calculate the invariants of $Y$:
$$\chi(Y) = 2 + \frac{1}{2}(K_S + \frac{E_4 + Z''}{2})\frac{E_4 +
Z''}{2} = 2 + \frac{1}{2}(1 + \frac{E_4^2 + Z''^2 + 2E_4Z''}{4})$$
$$p_g(Y) = h^0(S, 2K_S - L_3 - Z') = h^0(S, L_3 + Z' + E_4 + Z'')
\geq 2.$$

\begin{Claim} $\chi(Y) \leq 2$ and $q(Y)\geq 1$.\end{Claim}
\begin{proof}
We only need to show that $\frac{1}{2}(1 + \frac{E_4^2 + Z''^2 +
2E_4Z''}{4}) < 1$, i.e., $E_4^2 + Z''^2 + 2E_4Z'' < 4$. If $Z'' =
0$, then we are done. So we may assume $Z'' \neq 0$, then
$Z''^2 = -2$, thus it suffices to show $E_4Z''< 3 -
\frac{E_4^2}{2}$. From the table above, it follows from the observation that $E_4Z'' \leq \frac{E_4Z}{b}$.
\end{proof}

We denote by $\alpha: Y \rightarrow B$ the Albanese pencil. By proposition~\ref{B.i}, we get a fibration $g: S \rightarrow
\mathbb{P}^1$ and a double cover $\pi': B \rightarrow
\mathbb{P}^1$ such that $g\circ \pi = \pi' \circ \alpha$. Checking
that $(E_4 + Z'')\theta_i < 4$, then using the Riemann-Hurwitz formula,
$\pi^{-1}\theta_i$ is composed with one or two rational curves (maybe singular) thus contained in some fibers of $\alpha$,
hence $\theta_i$ is contained in one fiber of $g$. Since $\rho(\bar{S}) =
\rho(\Sigma)$, by Remark~\ref{rpnf}, we may assume $g$ is induced
by the pencil $|L_1| = \phi^*|\bar{l_1}|$. Since $\pi': B
\rightarrow \mathbb{P}^1$ is branched over at least 4 points, so
$\bar{g}: \bar{S} \rightarrow \mathbb{P}^1$ has at least 3 double
fibers $2\bar{M_1}, 2\bar{M_2}, 2\bar{M_3}$ which does not contain
$\bar{E_4}$ as a component. Since every fiber of $\bar{g'}$ is
reduced, $\bar{\phi}$ is ramified along $\bar{M_1}, \bar{M_2},
\bar{M_3}$.
\begin{Claim} With the assumptions above, $E_2$ is non-reduced. \end{Claim}
\begin{proof}
First remark that $\pi': B \rightarrow \mathbb{P}^1$ is branched
over the point $g(E_4) = g(E_2)$. Let
$Z = B\times_{\mathbb{P}^1}S$, and denote by $p: Z \rightarrow S$ the projection. If $E_2$ is reduced, then the map $p$ is
branched along $E_2$, and $Z$ is normal along the locus over $E_2$.
Note that there is a natural birational morphism $h: Y\rightarrow
Z$ such that $\pi = p\circ h$, so the map $\pi: Y\rightarrow S$ is
also branched over $E_2$. But the branch locus of $\pi$ does not contain $E_2$, and we are done.
\end{proof}

By the claim, we may assume $E_2= 2E_2'$, then $\bar{\phi}$ is
ramified along $\bar{E_2'}$. Denote the ramification
divisor of $\bar{\phi}$ by $R$,
and put $R_1 = \bar{M_1} + \bar{M_2} + \bar{M_3} + \bar{E_2'}$. Note that $M_i(\bar{\phi}^*\bar{e_1}) = 2,i=1,2,3$ and $\bar{E_2'}(\bar{\phi}^*\bar{e_1}) = \frac{4}{3}$ since $\bar{e_2}\bar{e_1} = \frac{2}{3}$. We
get a contradiction from Lemma~\ref{ramification} since $(R-R_1)\bar{E_1} = -\frac{4}{3} = (\bar{\phi}^*\bar{e_1})^2$.

In conclusion, we prove $\hat{P} \neq \hat{P}_4$.

\subsubsection{The case $\hat{P} = \hat{P}_3$}\label{3}

\begin{Notation} We denote by $l_4$ the strict transform of a
general line through $P_4$, by $l_1$ the strict transform of a
general line through $P_1$. Let $Q_1 = \psi(c)$ and $Q_2 =
\psi(e_1+ e_2)$. Denote by $L_4$ (resp.$L_1, E_3, E_4$) the strict
transforms of $\bar{l_4}$ (resp.$\bar{l_1}, \bar{e_3}, \bar{e_4}$)
with respect to $\phi$.\end{Notation}

Immediately, it follows that $l_4 \equiv l -e_4$, $l_1 \equiv l -
e_1 - e_2 - e_3$ and $c \equiv l - e_1 - 2e_2 - 3e_3$. So we have
\begin{equation}
\begin{split}
-K_{\hat{P}} &\equiv 3l - e_1 - 2e_2 - 3e_3 - e_4 \\
&\equiv 3l_4 +2e_4 - e_1 - 2e_2 - 3e_3\\
&\equiv 2l_4 + e_4 + c
\end{split}
\end{equation}
hence
\begin{equation}
\begin{split}
-K_{\Sigma}
&\equiv 3\bar{l_4} + 2\bar{e_4} - 3\bar{e_3}\\
&\equiv 2\bar{l_4} + \bar{e_4}
\end{split}
\end{equation}
Considering the pull-backs $$\psi^*\bar{l_4} = l_4 + \frac{1}{2}c,~~~
\psi^*\bar{l_1} = l_1 + \frac{2}{3}e_1+
\frac{1}{3}e_2,~~~\psi^*\bar{e_3} = e_3 + \frac{2}{3}e_2+
\frac{1}{3}e_1 + \frac{1}{2}c$$ we get
$$\bar{l_4}^2 = \frac{1}{2},~~~\bar{l_1}^2 = \frac{2}{3},~~~\bar{e_3}^2 = \frac{1}{6}.$$

Using the fact $\bar{l_1}^2 = \frac{2}{3}$, similar argument as in
Section~\ref{4} Claim~\ref{layoutof(-2)-curves} shows

\begin{Claim} $\phi^{-1}Q_2 = \theta_2 \cup \theta_3$ where $\theta_2$ and
$\theta_3$ are $(-2)$-curves such that $\theta_2\theta_3 = 1$.
\end{Claim}

Then there exists at most one $(-2)$-curve on $S$ except for
$\theta_2,\theta_3$ since $\rho(S) - \rho(\Sigma) = 3$.

\begin{itemize}
\item[Case 1:]{There's no $(-2)$-curve aside from
$\theta_2,\theta_3$, consequently $\phi^{-1}Q_1$ is composed with two
points.}
\end{itemize}

We write $\phi^*3\bar{e_3} = 3E_3 + Z$ where $Z =
a\theta_2 + b\theta_3$ with $a,b$ are positive integers, and
assume $a \geq b$.

Note that in this case, $\phi^*\bar{l_4}, \phi^*\bar{e_4}$ are Cartier divisors that do not contain $(-2)$-curves.
Then by $-K_{\Sigma} \equiv 3\bar{l_4} + 2\bar{e_4} - 3\bar{e_3}$,
we have $$2K_S \equiv 3L_4 + 2\phi^*\bar{e_4} - 3E_3 - Z$$

Subcase 1: $E_3$ is reduced.

In this case, by Lemma~\ref{key1}, since $(\phi^*\bar{e_3})^2 = \frac{2}{3}$ and $E_3^2 \leq (\phi^*\bar{e_3})^2$, it is possible that
$E_3^2=0$~or~$-2$~or~$-4$~or~$-6$. Combining the two formulas
$(\phi^*3\bar{e_3})^2 = 6$ and $(\phi^*\bar{e_3})\theta_2 =
(\phi^*\bar{e_3})\theta_3 = 0$, we get that:
$$9E_3^2 + 6aE_3\theta_2 + 6bE_3\theta_3 - 2a^2 + 2ab - 2b^2 = 6$$
$$3E_3\theta_2 - 2a + b = 3E_3\theta_3 + a - 2b = 0$$
Resolving these
equations, we get that either
$$E_3^2 = 0, ~~Z = 2\theta_2 + \theta_3,~~E_3\theta_2 = 1,
~~E_3\theta_3 = 0
$$ or
$$E_3^2 = -2, ~~Z = 4\theta_2 +
2\theta_3, ~~E_3\theta_2 = 2, ~~E_3\theta_3 = 0$$ or
$$E_3^2 = -4, ~~Z = 5\theta_2 +
4\theta_3~~E_3\theta_2 = 2, ~~E_3\theta_3 = 1.$$

If we write $Z = 2Z' +
Z''$ where $Z''$ is reduced, then we have
$$2(K_S - L_4 - \phi^*\bar{e_4} + 2E_3 + Z' + Z'')
\equiv L_4 + E_3 + Z''$$ By the relation, we get a double cover $\pi: Y
\rightarrow S$. Note that $L_4^2= 4\bar{l_4}^2 = 2$, $L_4Z = 0$ and
$L_4E_3 = 2$. Checking that $[-\phi^*\bar{e_3}] = [-E_3 - \frac{Z}{3}] \leq -E_3 - Z'$, then by use of Formula \ref{tool} and Lemma~\ref{pullback}, we calculate the
invariants of $Y$ as follows:
\begin{equation}
\begin{split}
\chi(Y) &= 2 + \frac{(K + \frac{L_4 + E_3 + Z''}{2}) (\frac{L_4 +
E_3 + Z''}{2})}{2}\\
 & = 2 + \frac{3 + \frac{L_4^2 + 2L_4E_3 + Z''^2 + E_3^2
+ 2E_3Z''}{4}}{2} = 4\\
p_g(Y) &= h^0(S, 2K_S - L_4 - \phi^*\bar{e_4} + 2E_3 + Z' + Z'')= h^0(S, 2L_4 + \phi^*\bar{e_4} - E_3 - Z') \\
&\geq h^0(\mathcal{O}_\Sigma(2\bar{l_4} + \bar{e_4} - \bar{e_3}))
\geq h^0(\hat{P}, 2l - e_1 - e_2 - e_3 - e_4) \geq 4,
\end{split}
\end{equation}
thus $q(Y) \geq 1$. By Proposition~\ref{A.i}, we obtain
\[\begin{CD}
Y          @>\pi>>          S \\
@V\alpha VV                @VgVV \\
B        @>\pi'>>      \mathbb{P}^1
\end{CD} \]
where $\alpha: Y \rightarrow B$ is the Albanese pencil and $\pi':B
\rightarrow \mathbb{P}^1$ is a double cover. Note that $L_4$ is contained in
one fiber since it is branched over by $\pi: Y \rightarrow S$.
This contradicts $L_4^2 = 2$, so this case does not occur.

Subcase 2: $E_3$ is non-reduced.

By Lemma \ref{key1}, we can assume $E_3 = 2E_3'$. Then we get:
$$2(K_S - L_4 - \phi^*\bar{e_4} + 3E_3' + Z' + Z'')
\equiv L_4 + Z''.$$ Considering the double cover $\pi: Y \rightarrow S$ given by the
relation, and calculating the invariants of $Y$, we obtain:
\begin{equation}
\begin{split}
\chi(Y) &= 2 + \frac{(K + \frac{L_4 + Z''}{2}) (\frac{L_4 +
Z''}{2})}{2} \\
& = 2 + \frac{2 + \frac{L_4^2 + Z''^2}{4}}{2} \leq 3\\
p_g(Y) &= h^0(S, 2K_S - L_4 - \phi^*\bar{e_4} + 3E_3' + Z' + Z'')\\
&= h^0(S, 2L_4 + \phi^*\bar{e_4} - 3E_3' - Z') \\
&\geq h^0(\mathcal{O}_\Sigma(2\bar{l_4} + \bar{e_4} - 2\bar{e_3}))
\geq h^0(\hat{P}, 2l - e_1 - 2e_2 - 2e_3 - e_4) =3,
\end{split}
\end{equation}
and thus $q(Y) \geq 1$. By similar argument as above, we get a contradiction.

\begin{itemize}
\item[Case 2:]{There exists another $(-2)$-curve on $S$ except
for $\theta_2,\theta_3$.}
\end{itemize}

Since there are 3 $(-2)$-curves on $S$, $\rho(\bar{S}) =
\rho(\Sigma) = 2$. By $-K_{\Sigma} \equiv 2\bar{l_4} + \bar{e_4}$,
we have
$$2K_S \equiv 2L_4 + E_4 + Z$$
where $Z$ is zero or supported on exactly one $(-2)$-curve since $\bar{l_4}$ and $\bar{e_4}$ do not contain the $A_2$-singularity $Q_2$. We may write $Z=2Z'
+ Z''$ where $Z''$ is reduced. No matter whether $E_4$ is reduced
or not, by the double covering trick, arguing as in the proof of Lemma \ref{reducedornot} and Case 1 in Section \ref{1}, it is easy to show that $S$
has a fibration such that every $(-2)$-curve is contained in one
fiber. We omit the details. Then by Remark~\ref{rpnf}, there exists a fibration of
$\hat{P}$ such that every $(-2)$-curve is contained in one fiber.
However, this contradicts Proposition~\ref{fibers}.

In conclusion we prove $\hat{P} \neq \hat{P}_3$

\subsubsection{The case $\hat{P} = \hat{P}_5$}\label{5}

\begin{Notation} Denote by $l_2$ the strict transform of a general
line through $P_2$. Let $Q = \psi(c + e_2 +e_3)$ be the
$A_3$-singularity on $\Sigma$. Denote by $L_2$ (resp.$E_1, E_4$) the
strict transforms of $\bar{l_2}$ (resp.$\bar{e_1}, \bar{e_4}$)
with respect to $\phi$.
\end{Notation}

It follows that $l_2 \equiv l -e_2 - e_3 -e_4$, $c \equiv l - e_1
- e_2 - 2e_3 - 2e_4$, and then
$$-K_{\hat{P}} \equiv 3l - e_1 - e_2 - 2e_3 - 3e_4
\equiv 2l_2 + e_4 + c + 2e_2 + 2e_3$$ hence
$$-K_{\Sigma} \equiv 2\bar{l_2} + \bar{e_4}$$
Considering the pull-backs
$$\psi^*\bar{l_2} = l_2 + \frac{3}{4}e_2+ \frac{1}{2}e_3 +
\frac{1}{4}c
\equiv \frac{5}{4}l - \frac{1}{4}e_1 - \frac{1}{2}e_2 - e_3-
\frac{3}{2}e_4$$
$$\psi^*\bar{e_4} = e_4 + e_3+ \frac{1}{2}(e_2
+ c) \equiv \frac{1}{2}(l - e_1)$$
$$\psi^*\bar{e_1} = e_1 + \frac{3}{4}c+ \frac{1}{2}e_3 +
\frac{1}{4}e_2 \equiv \frac{3}{4}l + \frac{1}{4}e_1 -
\frac{1}{2}e_2 - e_3 - \frac{3}{2}e_4 \equiv \bar{l_2} -
\bar{e_4},$$ we get
$$\bar{l_2}^2 = \frac{3}{4},~~\bar{e_4}^2 = 0,~~\bar{l_2}\bar{e_1} =\frac{1}{4},~~ \bar{e_4}\bar{e_1} = \frac{1}{2},~~
\bar{l_2}\bar{e_4} = \frac{1}{2},~~\bar{e_1}^2 = -\frac{1}{4}.$$

Since $\rho(S) - \rho(\Sigma) = 3$, $S$ has at most three $(-2)$-curves, precisely we have the following claim.
\begin{Claim}\label{nofc} $\phi^{-1}(Q)$ is composed with either three $(-2)$-curves which we denote by $\theta_1,\theta_2, \theta_3$
such that $\theta_1\theta_2 = \theta_2\theta_3 = 1$ or 2 disjoint
$(-2)-curves$.
\end{Claim}
\begin{proof}
Set $\phi^*\bar{e_1} = E_1 + x\theta_1 + y\theta_2 +
z\theta_4$ and $\phi^*\bar{e_4} = D_4 + u\theta_1 + v\theta_2
+ w\theta_4$ where $\theta_i$ is either zero or a $(-2)$-curve
mapped to $Q$. Note that if $\theta_i$ is non-zero, then so is its coefficient appearing in $\phi^*\bar{e_j},j=1,4$. We set the coefficient of $\theta_i$ to be zero if $\theta_i =0$, so we can assume $\theta_i^2 = -2$ in the following calculations.
Set: $E_1\theta_i = k_i,E_4\theta_i = n_i$~for
$i=1,2,3$;~$\theta_1\theta_3 = 0,\theta_1\theta_2 =
\alpha,\theta_2\theta_3 = \beta$ where $\alpha,\beta=0~or~1$.

The equations $(\phi^*\bar{e_1})\theta_i = 0$ and $(\phi^*\bar{e_4})\theta_i =
0$ yield
\begin{align*}
2x         -\alpha y                &= k_1 &   & &    2u          -\alpha v          &= n_1\\
-\alpha x  +2y          -\beta z    &= k_2 &and& &    -\alpha u  + 2v        -\beta w  &= n_2\\
        -\beta y    +2z             &= k_3 &   & &            \beta v   +2w       &= n_3
\end{align*}

If $\alpha = \beta = 0$, then the equations $(\phi^*\bar{e_1})^2 = -1,
(\phi^*\bar{e_4})^2 = 0, (\phi^*\bar{e_1})(\phi^*\bar{e_4}) = 2$ yield
\begin{equation}
\begin{split}
E_1^2 + \frac{1}{2}(k_1^2 + k_2^2 + k_3^2) &= -1\\
E_4^2 + \frac{1}{2}(n_1^2 + n_2^2 + n_3^2) &= 0\\
E_1E_4 + \frac{1}{2}(n_1k_1 + n_2k_2 + n_3k_3) &= 2
\end{split}
\end{equation}
Since $E_1^2$ is even, the first equation implies two of
$k_1,k_2,k_3$ are odds, we may assume $k_1,k_2$ are odds, so
$\theta_1,\theta_2$ are non-zero, and $n_1, n_2>0$ since then $u,v>0$. Since $E_4^2$
is even, the second equation implies that $n_1, n_2, n_3$ are even
simultaneously. Since $E_1E_4 \geq 0$, the third equation implies
$k_1 = k_2 = 1, n_1 = n_2 = 2, k_3 = n_3 = 0$. Then we conclude
that that $\phi^{-1}Q$ is two disjoint $(-2)$-curves.

If $\alpha =1, \beta = 0$, then $(\phi^*\bar{e_1})^2 = -1$ yields:
$$E_1^2 + \frac{2}{3}(k_1^2 + k_2^2 + k_1k_2) + \frac{1}{2}k_3^2=
-1$$ To guarantee the left hand side is an integer, $k_3$ should be even,
but then the left side of the equation is even.

If $\alpha = \beta = 1$, then $\phi^{-1}(Q)$ is composed with 3
$(-2)$-curves and $\theta_1\theta_2 = \theta_2\theta_3 = 1$, and we are done.
\end{proof}

\begin{itemize}
\item[Case 1:]{$\phi^{-1}(Q)$ is composed with 3 $(-2)$-curves.}
\end{itemize}

By $-K_{\Sigma} \equiv 2\bar{l_2} + \bar{e_4}$, we may write $2K_S
= 2L_2 + E_4 + Z$, and write $Z = 2Z' + Z''$ where $Z''$ is reduced.

First if $E_4$ is non-reduced, we may write $E_4 =
2E_4'$. Consider the double cover of $S$ given by the relation $2(K_S - L_2 - E-4' - Z') \equiv Z''$.
Note that $\bar{E_4} = 2\bar{E_4}'$, so $\bar{E_4}'\leq \bar{R}$. Then arguing as Section \ref{1} case 1, we get a contradiction.

Now we assume that $E_4$ is reduced. Since $\bar{e_4}^2 = 0$ and $\phi^*\bar{e_4}$ contains $(-2)$-curves, so $E_4^2 < 0$. Then Lemma~\ref{key1} tells that
\begin{equation}\label{51}
E_4^2 = -2 ~or~
-4~or~-6.
\end{equation}
Assume $Z = a\theta_1 + b\theta_2 + c\theta_3$ and $a\geq c$. Then equation $2K_S
= 2L_2 + E_4 + Z$ intersecting $E_4$ yields $0< E_4Z = 4 - E_4^2 - 2L_2E_4\leq 10$ which implies
\begin{equation}\label{52}
L_2E_4 < 5, min\{a,b,c\}\leq 10,a\geq c
\end{equation}
By $Z \equiv 2K_S- 2L_2 - E_4$, calculating $Z^2$, we get that
\begin{equation}\label{53}
-\frac{Z^2}{2} = a^2 +
b^2 + c^2 - ab - bc = 10 - 2L_2^2 - 2L_2E_4 - \frac{E_4^2}{2}
\end{equation}
And $Z\theta_i \leq 0,i=1,2,3$ yields that
\begin{equation}\label{54}
2a\geq
b, 2b\geq a + c, 2c \geq b
\end{equation}

Note that since $L_2^2 < 4\bar{l_2}^2 = 3$, we get
\begin{equation}\label{55}
L^2 = 0~or~2
\end{equation}

Then by \ref{51},\ref{52},\ref{53},\ref{54} and \ref{55},
going a computer program, we get the following possibilities:
\begin{align*}
&a  &  &b  &   &c  &  &L_2^2  &  &L_2E_4    &  &E_4^2    &  &E_4Z \\
&1  &  &1  &   &1  &  &2      &  &4         &  &-6       &  &2     \\
&1  &  &2  &   &1  &  &2      &  &3         &  &-4       &  &2      \\
&2  &  &2  &   &1  &  &2      &  &2         &  &-2       &  &2      \\
&2  &  &3  &   &2  &  &2      &  &1         &  &-2       &  &4      \\
&3  &  &2  &   &1  &  &0      &  &3         &  &-4       &  &2      \\
&3  &  &2  &   &1  &  &2      &  &1         &  &-4       &  &6      \\
&3  &  &3  &   &2  &  &2      &  &0         &  &-2       &  &6      \\
&3  &  &3  &   &2  &  &0      &  &2         &  &-2       &  &2      \\
&3  &  &3  &   &3  &  &0      &  &1         &  &-2       &  &4      \\
&3  &  &4  &   &2  &  &0      &  &1         &  &-2       &  &4      \\
&3  &  &4  &   &3  &  &0      &  &1         &  &-4       &  &2      \\
&4  &  &3  &   &2  &  &0      &  &0         &  &-2       &  &6      \\
&2  &  &4  &   &2  &  &0      &  &2         &  &-4       &  &4      \\
&4  &  &4  &   &2  &  &0      &  &0         &  &-4       &  &8      \\
&2  &  &4  &   &2  &  &2      &  &0         &  &-4       &  &8      \\
&3  &  &2  &   &1  &  &3      &  &0         &  &-4       &  &8      \\
&2  &  &2  &   &1  &  &2      &  &3         &  &-6       &  &4      \\
&3  &  &4  &   &2  &  &0      &  &2         &  &-6       &  &6     \\
&3  &  &4  &   &2  &  &2      &  &0         &  &-6       &  &10      \\
&4  &  &4  &   &3  &  &0      &  &0         &  &-6       &  &10      \\
\end{align*}
Consider the double cover $\pi:Y \rightarrow S$
branched over $E_4$ and $Z''$ given by $2(K_S - L_2 - Z') \equiv
E_4 + Z''$ and check that $E_4+Z''$ has at most negligible singularities.
So $Y$ has at most canonical singularities, applying Formula \ref{tool}, we get
$$\chi(Y) = 2 + \frac{1}{2}(K_S +
\frac{E_4 + Z''}{2})\frac{E_4 + Z''}{2} = 2 + \frac{1}{2}(1 +
\frac{E_4^2 + Z''^2 + 2E_4Z''}{2}),$$ $$p_g(Y) = h^0(S, 2K_S - L_2
- Z') = h^0(S, L_2 + Z' + E_4 + Z'') \geq 2.$$

\begin{Claim} $\chi(Y) \leq 2$ and $q(Y)\geq 1$.\end{Claim}
\begin{proof}
We only need to show that $\frac{1}{2}(1 + \frac{E_4^2 + Z''^2 +
2E_4Z''}{4}) < 1$, i.e., $E_4^2 + Z''^2 + 2E_4Z'' < 4$. If $Z'' =
0$, then we are done. Now we assume $Z'' \neq 0$, then $Z''^2
= -2~or~-4$, so it suffices to show that $E_4Z''< 3 -
\frac{E_4^2}{2}$. Observing that $E_4Z'' \leq \frac{E_4Z}{min\{a,b,c\}}$,
we exclude all the possibilities listed above except the case $Z = 3\theta_1 + 2\theta_2 + \theta_3$, $E_4Z =
6$ and $Z'' = \theta_1 + \theta_3$. Immediately we get
$Z\theta_3=0$ which implies that $E_4\theta_3 = 0$. By $E_4Z = 6$,
we have $E_4\theta_1 \leq 2$, hence $E_4Z'' \leq 2 <4$. Then the
claim holds true.
\end{proof}
Denote by $\alpha: Y \rightarrow B$ the Albanese map to the
image.
By Proposition~\ref{A.i}, we get the following commutative
diagram:
\[\begin{CD}
Y          @>\pi>>          S \\
@V\alpha VV                @VgVV \\
B        @>\pi'>>      \mathbb{P}^1
\end{CD} \]
where $\pi':B
\rightarrow \mathbb{P}^1$ is a double cover. Checking that $(E_4 + Z'')\theta_i \leq 3$, we can see the inverse
image $\phi^{-1}\theta_i$ is composed with 1 or 2 rational curves
for $i = 1,2,3$, hence $\theta_i$ is contained in one
fiber of $g$. By Remark~\ref{rpnf}, applying
Proposition~\ref{fibers}, the fibration $g: S \rightarrow
\mathbb{P}^1$ is given by the pencil $|F| = \phi^*|\bar{l_1}|$.
Let $P = g(E_4)$, then $g^*P_4 = 2E_4 + D$ since $2\bar{e_4}
\equiv \bar{l_1} - \bar{e_1}$. Let $Z = B\times_{\mathbb{P}^1}S$,
and let $\tilde{Z}\rightarrow Z$ be the normalization. Since $g^*P
= 2E_4 + D$ and $\pi'$ is branched over $P$, the map $p: \tilde{Z}
\rightarrow S$ is not branched over $E_4$. There is a morphism $h:
Y\rightarrow \tilde{Z}$ such that $\pi = p\circ h$. Note that $h:
Y\rightarrow Z$ is birational morphism since both $\pi$ and $p$
have the same degree 2. So the map $\pi: Y\rightarrow S$ is not
branched over $E_4$ which contradicts the construction. Therefore
this case does not occur.

\begin{itemize}
\item[Case 2:]{$\phi^{-1}Q$ is composed with two disjoint $(-2)$-curves.}
\end{itemize}

Assume that $\phi^{-1}Q = \theta_1 \cup \theta_2$ where $\theta_1
$ and $\theta_2$ are two disjoint $(-2)$-curves. Remark that there might be another
$(-2)$-curve $\theta_3$ disjoint from $\theta_1,\theta_2$ mapped to $\bar{e_1}$ via $\phi$.

Subcase 1: There is no such a $(-2)$-curve $\theta_3$ that is disjoint from $\theta_1,\theta_2$ and mapped to $\bar{e_1}$ via $\phi$.

By the proof of the claim \ref{nofc}, we have $\phi^*\bar{e_1} =
E_1 + \frac{1}{2}(\theta_1 + \theta_2)$ where $a$ is an integer
and is set to be zero if there is no such a $(-2)$-curve, $\phi^*\bar{l_2} = L_2 +
\frac{1}{2}(\theta_1 + \theta_2)$, $E_1^2 = -2$, $L_2^2 = 2$ and
$E_1L_2 = 0$, in particular, $E_1$ is reduced.

By $-K_{\hat{P}} \equiv 3l - e_1 - e_2 - 2e_3 - 3e_4 \equiv 3l_2 -
e_1 + 2e_2 + e_3,$ we have $-K_{\Sigma} \equiv 3\bar{l_2} -
\bar{e_1}$. So we have $2K_S \equiv 3L_2 - E_1 +  \theta_1 +
 \theta_2$, thus $2(K_S - L_2 +E_1) \equiv L_2 + E_1 +  \theta_1 +
  \theta_2$. In turn we get a double cover $\pi: Y \rightarrow S$. Similarly, check that $Y$ has at most canonical singularities.
Calculating the invariants of $Y$, we obtain
$$\chi(Y) = 4,~~p_g(Y) \geq 4,~~q(Y)\geq 1.$$ Note that $p_g(Y) \geq 4$ is due to $p_g(Y) \geq h^0(\mathcal{O}_\Sigma(2\bar{l_2})) = h^0(\hat{P}, 2l - e_2 - e_3 -2e_4) = 4$. By
Proposition~\ref{A.i}, we get the following commutative diagram:
\[\begin{CD}
Y          @>\pi>>          S \\
@V\alpha VV                @VgVV \\
B        @>\pi'>>      \mathbb{P}^1
\end{CD} \]
where $\alpha: Y \rightarrow B$ is the Albanese map and $\pi':B
\rightarrow \mathbb{P}^1$ is a double cover. $L_2$ must be
contained in one fiber, but this contradicts $L_2^2 = 2 > 0$. So
this case does not occur.

Subcase 2: There is a $(-2)$-curve $\theta_3$ disjoint from $\theta_1,\theta_2$ mapped to $\bar{e_1}$ via $\phi$.

Note that $\theta_3$ is not mapped to $\bar{e_4}$ and
$\bar{l_2}$, so
$\phi^*\bar{e_4} = E_4 + (\theta_1 + \theta_2)$ and
$\phi^*\bar{l_2} = L_2 + \frac{1}{2}(\theta_1 + \theta_2)$. Therefore we get the ralation $2K_S \equiv 2L_2 + E_4 + 2(\theta_1 + \theta_2)$.
Then similar argument as in Case 1 shows that this case
does not occur.

In conclusion, we prove that $\hat{P} \neq \hat{P}_5$

\subsubsection{The case $\hat{P} = \hat{P}_6$}\label{6}

\begin{Notation}Let $Q$ be the $A_4$-singularity on $\Sigma$. We denote
by $l_1$ the strict transform of a general line through $P_1$ and
by $c$ the strict transform of the line through $P_1, P_2, P_3$.
\end{Notation}

Immediately it follows that $l_1 \equiv l  - e_1 - e_2 - e_3 -e_4$ and $c \equiv l -
e_1 - 2e_2 - 3e_3 - 3e_4$.
Then we have: $$-K_{\hat{P}} \equiv 3l - e_1 - 2e_2 -3e_3 -4e_4
\equiv 2l_1 + e_4 + c +2e_1 + 2e_2 + 2e_3$$ and
$$-K_{\Sigma}
\equiv 2\bar{l_1} + \bar{e_4}
$$
About the $(-2)$-curves on $S$, we have the following claim:
\begin{Claim} $\phi^{-1}Q$ is composed with 4 $(-2)$-curves: $\theta_1,\theta_2, \theta_3,\theta_4$ such that
$\theta_i\theta_{i+1} = 1$~for~$i=1,2,3$.
\end{Claim}
\begin{proof}
Considering the self-intersection $\bar{l_1}^2 = \frac{4}{5}$,
then the claim follows from similar argument as in the proof of
Claim~\ref{layoutof(-2)-curves} in Section~\ref{4}.
\end{proof}

By the claim, $\hat{P}$ and $S$ has the same Picard
number. Denote by $L_1, E_4$ the strict
transforms of $\bar{l_1}, \bar{e_4}$ respectively. We can write
$2K_S = 2L_1 + E_4 + Z$, and assume $Z = a\theta_1+ b\theta_2 +
c\theta_3 + d\theta_4$ where $a,b,c,d$ are positive integers and $a\geq d$.

\begin{Claim} There exists a fibration $g: S \rightarrow \mathbb{P}^1$ such that
every $(-2)$-curve is contained in a fiber.\end{Claim}
\begin{proof}
Case 1: $E_4$ is reduced.

Arguing as in Section \ref{5} Case 1, for the readers' convenience, we list all the possibilities for $Z$ and the the
intersection numbers of the divisors involved in the proof:
\begin{align*}
&a  &  &b  &   &c  &  &d  &  &L_1^2  &  &L_1E_4    &  &E_4^2    &  &E_4Z \\
&1  &  &2  &   &2  &  &1  &  &2      &  &3         &  &-4       &  &2     \\
&2  &  &2  &   &2  &  &1  &  &2      &  &2         &  &-2       &  &2     \\
&2  &  &2  &   &2  &  &2  &  &2      &  &2         &  &-4       &  &4     \\
&2  &  &3  &   &2  &  &1  &  &2      &  &2         &  &-4       &  &4     \\
&2  &  &3  &   &3  &  &2  &  &2      &  &1         &  &-2       &  &4     \\
&2  &  &3  &   &4  &  &2  &  &0      &  &2         &  &-2       &  &2     \\
&2  &  &3  &   &4  &  &2  &  &2      &  &0         &  &-2       &  &6     \\
&2  &  &4  &   &3  &  &2  &  &2      &  &0         &  &-2       &  &6     \\
&2  &  &4  &   &4  &  &2  &  &0      &  &2         &  &-4       &  &4     \\
&2  &  &4  &   &4  &  &2  &  &2      &  &0         &  &-4       &  &8     \\
&3  &  &3  &   &2  &  &1  &  &0      &  &3         &  &-4       &  &2     \\
&3  &  &3  &   &2  &  &1  &  &2      &  &1         &  &-4       &  &6     \\
&3  &  &3  &   &3  &  &2  &  &0      &  &2         &  &-2       &  &2     \\
&3  &  &3  &   &3  &  &3  &  &0      &  &1         &  &-2       &  &4     \\
&3  &  &4  &   &3  &  &2  &  &2      &  &0         &  &-4       &  &8     \\
&3  &  &4  &   &3  &  &2  &  &0      &  &2         &  &-4       &  &4     \\
&3  &  &4  &   &4  &  &2  &  &0      &  &1         &  &-2       &  &4     \\
&3  &  &4  &   &4  &  &3  &  &0      &  &1         &  &-4       &  &6     \\
&3  &  &4  &   &5  &  &3  &  &0      &  &0         &  &-4       &  &8     \\
&3  &  &5  &   &4  &  &2  &  &0      &  &0         &  &-2       &  &6     \\
&3  &  &5  &   &4  &  &3  &  &0      &  &0         &  &-4       &  &8     \\
&4  &  &3  &   &2  &  &1  &  &0      &  &1         &  &-4       &  &6     \\
&4  &  &4  &   &3  &  &2  &  &0      &  &0         &  &-2       &  &6     \\
&4  &  &4  &   &4  &  &2  &  &0      &  &0         &  &-4       &  &8     \\
&2  &  &3  &   &3  &  &2  &  &0      &  &4         &  &-6       &  &2     \\
&1  &  &1  &   &1  &  &1  &  &2      &  &4         &  &-6       &  &2     \\
&2  &  &2  &   &2  &  &1  &  &2      &  &3         &  &-6       &  &4     \\
&2  &  &3  &   &3  &  &2  &  &2      &  &2         &  &-6       &  &6     \\
&2  &  &3  &   &4  &  &2  &  &0      &  &3         &  &-6       &  &4     \\
&2  &  &3  &   &4  &  &2  &  &2      &  &1         &  &-6       &  &8     \\
&2  &  &4  &   &3  &  &2  &  &0      &  &3         &  &-6       &  &4     \\
&2  &  &4  &   &3  &  &2  &  &2      &  &1         &  &-6       &  &8     \\
&3  &  &3  &   &3  &  &2  &  &0      &  &3         &  &-6       &  &4     \\
&3  &  &3  &   &3  &  &2  &  &2      &  &1         &  &-6       &  &8     \\
&3  &  &3  &   &3  &  &3  &  &0      &  &2         &  &-6       &  &6     \\
&3  &  &3  &   &3  &  &3  &  &2      &  &0         &  &-6       &  &10     \\
&3  &  &4  &   &4  &  &2  &  &0      &  &2         &  &-6       &  &6    \\
&3  &  &4  &   &4  &  &2  &  &2      &  &0         &  &-6       &  &10     \\
&3  &  &5  &   &4  &  &2  &  &0      &  &1         &  &-6       &  &8    \\
&3  &  &5  &   &5  &  &3  &  &0      &  &0         &  &-6       &  &10     \\
&4  &  &4  &   &3  &  &2  &  &0      &  &1         &  &-6       &  &8    \\
&4  &  &4  &   &4  &  &3  &  &0      &  &0         &  &-6       &  &10     \\
&4  &  &5  &   &4  &  &2  &  &0      &  &0         &  &-6       &  &10    \\
\end{align*}
Then going similar progress as in Section~\ref{5} Case 1,
we get a fibration $g: S \rightarrow \mathbb{P}^1$ such that
every $(-2)$-curve is contained in a fiber.

Case 2: $E_4$ is non-reduced.

Considering the relation $2K_S = 2L_1 + E_4 + Z$, using double cover trick, it is easy to prove that there exists such a
fibration.
\end{proof}
By the claim, Remark~\ref{rpnf} tells that there is a fibration
$\hat{g}: \hat{P} \rightarrow \mathbb{P}^1$ of $\hat{P}$ such that
every $(-2)$-curve is contained in one fiber. Then a contradiction follows from
Proposition~\ref{fibers}, and we are done.

\subsubsection{Conclusion}
Finally, we complete the proof of Theorem \ref{Theorem B}.


\begin{thebibliography}{1dffs}
\bibitem[BC]{BC} I.Bauer, F.Catanese, Burniat surfaces II: secondary Burniat surfaces
form three connected components of the moduli
space, Invent.Math. 180 (2010), 559¨C588.
\bibitem[BPV]{BPV} W. Barth, C. Peters, A. Van de Ven, Compact complex
surfaces, Ergebnisse der Mathematik und ihrer Grenzgebiete, 3.
Folge, Band 4, Springer-Verlag, Berlin (1984).
\bibitem[Bu]{Bu} P. Burniat, Sur les surfaces de genre $P_{12} > 0$, Ann. Mat.
Pura Appl., IV Ser., 71 (1966), 1-24.
\bibitem[Cat]{Cat} F. Catanese, Singular bidouble covers and the construction
of interesting algebraic surfaces, in Algebraic Geometry:
Hirzebruch 70, A.M.S. Contemporary Mathematics, vol. 241 (1999),
97-120.
\bibitem[CC]{CC} F. Catanese, C. Ciliberto, On the irregularity of cyclic
coverings of algebraic varieties, Geometry of Complex Projective
Varieties - Ce- traro (Italy), June 1990, (A. Lanteri, M.
Palleschi, D. C. Struppa eds.), Mediterranean Press (1993),
89¨C116.
\bibitem[Har]{Har} Hartshorne, Algebraic geometry, Springer Verlag, Berlin (1987).
\bibitem[M]{M} M. Mendes Lopes, The degree of the generators of the canonical
ring of surfaces of general type with $p_g = 0$, Arch. Math. 69
(1997), 435-440.
\bibitem[MP1]{MP2} M. Mendes Lopes, R. Pardini, A connected component of the
moduli space of surfaces of general type with $p_g = 0$, Topology
40 (5) (2001), 977-991.
\bibitem[MP2]{MP3} M. Mendes Lopes, R. Pardini, The bicanonical map of surfaces
with $p_g = 0$ and $K^2 \geq 7$,  Bull. London Math. Soc. 33
(2001) 337-343.
\bibitem[MP3]{MP4} M. Mendes Lopes, R. Pardini,  A survey on the bicanonical
map of surfaces with $p_g=0$ and $K^2 \geq 2$, Algebraic Geometry.
A volume in memory of Paolo Francia, Walter de Gruyter, (2002)
(math.AG/0202187).
\bibitem[MP4]{MP5} M. Mendes Lopes, R. Pardini The classification of surfaces
with $p_g=0$, $K^2=6$ and non birational bicanonical map, Math.
Ann. 329 (2004) 535-552 (math.AG/0301138).
\bibitem[MP5]{MP6} M. Mendes Lopes, R. Pardini The degree of the bicanonical
map of the surface $p_g=0$, Proc. Amer. Math. Soc,135:55 (2007),
1279-1282 (math.AG/0505231v2).
\bibitem[Miy]{Miy} Y. Miyaoka, The maximal number of quotient singularities on
surfaces with given numerical invariants, Math. Ann., 268 (1973),
159-171.
\bibitem[Na]{Na} M. Nagata, On rational surfaces I, Mem. Coll. Sci., U. of Kyoto, ser.
A, vol. XXXII, Mathematics No. 3 (1960), 35-370.
\bibitem[Par1]{Par1} R. Pardini, Abelian covers of algebraic varieties, J. reine
angew. Math. 417 (1991), 191-213.
\bibitem[Par2]{Par2} R. Pardini, The classification of double planes of general
type with $K^2 = 8$ and $p_g = 0$, J. of Algebra 259 (2003),
95-118 (math.AG/0107100).
\bibitem[Pet]{Pet} C. Peters, On certain examples of surfaces with $p_g = 0$ due
to Burniat, Nagoya Math. J., Vol. 166 (1977), 109-119.
\bibitem[Re]{Re} I. Reider, Vector bundles of rank 2 and linear systems on
algebraic surfaces, Ann. of Math., 127 (1988), 309-316.
\bibitem[X1]{X1} G. Xiao, Finitude de l¡¯application bicanonique des surfaces
de type g¡äen¡äeral, Bull. Soc. Math. France, 113 (1985), 23-51.
\bibitem[X2]{X2} G. Xiao, Degree of the bicanonical map of a surface of
general type, Amer. J. of Math., 112 (5) (1990), 713-737.

\end{thebibliography}
\end{document}